\documentclass[11pt,reqno]{amsart}
\pdfoutput=1
\usepackage{amsmath}
\usepackage{amssymb}
\usepackage{amsthm}
\usepackage[foot]{amsaddr}
\usepackage[margin=1in]{geometry}
\usepackage[hidelinks]{hyperref}
\usepackage{bm}
\usepackage[Symbol]{upgreek}
\usepackage{mathtools}

\usepackage{lmodern}
\usepackage[T1]{fontenc}
\usepackage{microtype}

\DeclareMathOperator{\ord}{ord}

\DeclareMathOperator{\dmul}{dmul}
\DeclareMathOperator{\mul}{mul}

\DeclareMathOperator{\ddeg}{ddeg}

\def\d {\operatorname{d}}
\def\dh {\operatorname{dh}}
\def\ac {\operatorname{ac}}

\newcommand{\ca}{\mathcal}

\newcommand{\pconv}{\rightsquigarrow}
\newcommand{\prece}{\preccurlyeq}
\newcommand{\succe}{\succcurlyeq}
\newcommand{\x}{\times}
\newcommand{\f}{\mathfrak}
\newcommand{\xf}{\times\mathfrak}
\newcommand{\0}{\emptyset}
\newcommand{\wh}{\hat}
\newcommand{\hs}{\hspace{1cm}}
\newcommand{\ges}{\geqslant}
\newcommand{\les}{\leqslant}

\newcommand{\N}{\mathbb{N}}

\newcommand{\Q}{\mathbb{Q}}

\DeclareFontFamily{U}{fsy}{}
\DeclareFontShape{U}{fsy}{m}{n}{<->s*[.9]psyr}{}
\DeclareSymbolFont{der@m}{U}{fsy}{m}{n}
\DeclareMathSymbol{\der}{\mathord}{der@m}{182}
\DeclareFontFamily{OMS}{smallo}{}
\DeclareFontShape{OMS}{smallo}{m}{n}{<->s*[.65]cmsy10}{}
\DeclareSymbolFont{smallo@m}{OMS}{smallo}{m}{n}
\DeclareMathSymbol{\cao}{\mathord}{smallo@m}{79}

\newtheorem{lem}{Lemma}[section]
\newtheorem{prop}[lem]{Proposition}
\newtheorem{cor}[lem]{Corollary}
\newtheorem{thm}[lem]{Theorem}

\newtheorem{claim}{Claim}

\newtheorem*{maintechprop}{Proposition~\ref{maintechprop}}
\newtheorem*{mainlemmadiv}{Proposition~\ref{mainlemmadiv}}
\newtheorem*{maximmed}{Theorem~\ref{maximmed}}
\newtheorem*{dhdalgmax}{Theorem~\ref{dhdalgmax}}
\newtheorem*{dhensel}{Theorem~\ref{dhensel}}

\theoremstyle{definition}

\newtheorem*{defn}{Definition}
\newtheorem*{ass}{Assumption}

\numberwithin{claim}{lem}
\numberwithin{equation}{section}

\usepackage{etoolbox}

\makeatletter
\patchcmd{\@startsection}
  {\@afterindenttrue}
  {\@afterindentfalse}
  {}{}
\makeatother

\title[Differential-henselianity and maximality]{Differential-henselianity and maximality of\\ asymptotic valued differential fields}
\author{Nigel Pynn-Coates}
\address{Department of Mathematics, The Ohio State University, Columbus, OH, United States}
\email{\href{mailto:pynn-coates.1@osu.edu}{pynn-coates.1@osu.edu}}
\begin{document}

\begin{abstract}
We show that asymptotic (valued differential) fields have unique maximal immediate extensions.
Connecting this to differential-henselianity, we prove that any differential-henselian asymptotic field is differential-algebraically maximal, removing the assumption of monotonicity from a theorem of Aschenbrenner, van den Dries, and van der Hoeven \cite[Theorem~7.0.3]{adamtt}.
Finally, we use this result to show the existence and uniqueness of differential-henselizations of asymptotic fields.
\end{abstract}

\maketitle
%\tableofcontents

\section{Introduction}

The basic objects of this paper are valued differential fields (assumed throughout to have equicharacteristic~$0$) with small derivation and their \emph{extensions}, by which we mean valued differential field extensions with small derivation.
By Zorn, any valued differential field $K$ with small derivation has an immediate extension that is \emph{maximal} in the sense that it has no proper immediate extension.
By the main result of \cite{maximext}, such extensions are spherically complete, and thus maximal as valued fields in the usual sense.
Hence by Kaplansky \cite{kaplansky1} any two maximal immediate extensions of $K$ are isomorphic as valued fields over $K$.
If $K$ is additionally assumed to be monotone and to have linearly surjective differential residue field, then Aschenbrenner, van den Dries, and van der Hoeven showed in \cite[Theorem~7.4.3]{adamtt} that they are isomorphic as valued \emph{differential} fields over $K$ and conjectured in \cite{maximext} that monotonicity could be removed from this result.
Van den Dries and the present author showed that this conjecture holds when the value group is the union of its convex subgroups of finite archimedean rank \cite{dhfinrank}.
Here, we prove the conjecture when the valued differential field is asymptotic, a condition opposite in spirit to monotonicity.

\begin{maximmed}
If an asymptotic valued differential field $K$ has small derivation and linearly surjective differential residue field, then any two maximal immediate extensions of $K$ are isomorphic over~$K$.
\end{maximmed}
%\begin{thmint}\label{int:umie}
%If an asymptotic valued differential field $K$ has small derivation and linearly surjective differential residue field, then any two maximal immediate extensions of $K$ are isomorphic over $K$.
%\end{thmint}

We also show the uniqueness of differentially algebraic immediate extensions of $K$ that are differential-algebraically maximal.
This was likewise proved earlier in the monotone setting \cite[Theorem~7.4.3]{adamtt}, with the case of many constants essentially going back to Scanlon \cite{scanlon}.

Theorem~\ref{maximmed} has a consequence related to the topics discussed in Matusinski's survey \cite{matusinski} transposed to this setting.
Let $K$ be an asymptotic valued differential field with small derivation.
By \cite{maximext}, $K$ has an immediate extension $L$ that is asymptotic and spherically complete, so $L$ is isomorphic as a valued differential field to a Hahn field (also called a ``generalized series field'' in \cite{matusinski}) endowed with a derivation making it an asymptotic valued differential field with small derivation.
Hence if $K$ has linearly surjective differential residue field, then any other immediate differential Hahn field extension of $K$ that is asymptotic is also isomorphic to $L$ over $K$ by Theorem~\ref{maximmed}.

Since there is an equivalence between algebraic maximality and henselianity for valued fields of equicharacteristic 0, one might hope for a similar relationship in the setting of valued differential fields after adding the prefix ``differential,'' but it turns out to depend on the interaction between the valuation and the derivation.
In fact, any differential-algebraically maximal valued differential field with small derivation and linearly surjective differential residue field is differential-henselian \cite[Theorem~7.0.1]{adamtt}, but the converse fails, even in the setting of monotone fields with many constants \cite[example after Corollary~7.4.5]{adamtt}.
In contrast, we show that the converse holds in the case of few constants, removing entirely the monotonicity assumption in \cite[Theorem~7.0.3]{adamtt}:

\begin{dhdalgmax}
If $K$ is a valued differential field with small derivation and few constants that is differential-henselian, then $K$ is differential-algebraically maximal.
\end{dhdalgmax}
%\begin{thmint}\label{int:dhdam}
%If $K$ is a valued differential field with small derivation and few constants that is differential-henselian, then $K$ is differential-algebraically maximal.
%\end{thmint}

Note that by \cite[Lemmas~7.1.8 and 9.1.1]{adamtt}, any $K$ as in the theorem above is in fact asymptotic.
Using Theorem~\ref{dhdalgmax} we then prove the following generalization of \cite[Theorem~1.3]{dhfinrank}.

\begin{dhensel}
If $K$ is an asymptotic valued differential field with small derivation and linearly surjective differential residue field, then $K$ has a differential-henselization which is minimal, and so any two differential-henselizations of $K$ are isomorphic over $K$.
\end{dhensel}
%\begin{thmint}\label{int:dh}
%Suppose $K$ is an asymptotic valued differential field with small derivation and linearly surjective differential residue field.
%Then $K$ has a differential-henselization.
%Moreover, any two differential-henselizations of $K$ are isomorphic over $K$.
%\end{thmint}

\subsection{Basic definitions and notation}
To understand these results, we define the necessary conditions after setting up basic notation, which we keep close to that of \cite{adamtt}.
We let $d$, $i$, $m$, $n$, and $r$ range over $\N = \{0, 1, 2, \dots\}$.
Throughout, $K$ is a \emph{valued differential field}, that is, a field of characteristic $0$ together with a surjective map $v \colon K^{\x} \to \Gamma$, where $\Gamma$ is a (totally) ordered abelian group called the \emph{value group} of $K$, and a map $\der \colon K \to K$ satisfying $v(\Q^{\x})=\{0\}$ and, for all $x$, $y$ in their domain:
\begin{enumerate}
	\item[(V1)] $v(xy) = v(x)+v(y)$;
	\item[(V2)] $v(x+y) \ges \min\{v(x), v(y)\}$ whenever $x+y \neq 0$;
	\item[(D1)] $\der(x+y)=\der(x)+\der(y)$;
	\item[(D2)] $\der(xy)=x\der(y)+\der(x)y$.
\end{enumerate}
We introduce a symbol $\infty \notin \Gamma$ and extend the ordering to $\Gamma \cup \{\infty\}$ by $\infty>\Gamma$.
We also set $\infty+\gamma=\gamma+\infty = \infty+\infty \coloneqq \infty$ for all $\gamma \in \Gamma$, allowing us to extend $v$ to $K$ by setting $v(0) \coloneqq \infty$.
We also set $\Gamma^{\neq} \coloneqq \Gamma \setminus \{0\}$.
We let $\ca O \coloneqq \{ a \in K : v(a) \ges 0 \}$ be the \emph{valuation ring} of $K$ and $\cao \coloneqq \{ a \in K : v(a) > 0 \}$ its (unique) maximal ideal.
We also set $\ca O^{\neq} \coloneqq \ca O \setminus \{0\}$ and $\cao^{\neq} \coloneqq \cao \setminus \{0\}$.
Then $\bm k \coloneqq \ca O/\cao$ denotes the \emph{residue field} of $K$.
As it is often more intuitive, we define for $a, b \in K$:
\[\begin{array}{lc}
a \prece b\ \Leftrightarrow\ v(a)\ges v(b),\qquad a \prec b\ \Leftrightarrow\ v(a)> v(b),\\
a \asymp b\ \Leftrightarrow\ v(a)=v(b),\qquad  
  a\sim b\ \Leftrightarrow\ a-b \prec b.
\end{array}\]
Regarding the derivation, we usually write $a'$ for $\der(a)$ if the derivation is clear from context, and the \emph{field of constants} of $K$ is denoted by $C \coloneqq \{ a \in K : \der(a) = 0 \}$.
For another valued differential field $L$, we apply the subscript $L$ to these symbols; for example, $\ca O_L$ denotes the valuation ring of $L$.
Recall that a valued field extension $L$ of $K$ is said to be \emph{immediate} if $\bm k_L = \bm k$ and $\Gamma_L = \Gamma$, where we identify $\bm k$ with a subfield of $\bm k_L$ and $\Gamma$ with an ordered subgroup of $\Gamma_L$ in the usual way.
We call a pseudocauchy sequence a \emph{pc-sequence} and refer the reader to \cite[\S2.2 and \S3.2]{adamtt} for definitions and basic facts about them.

We let $K\{Y\} \coloneqq K[Y, Y', Y'', \dots]$ denote the ring of differential polynomials over $K$ (we hereinafter shorten ``differential-polynomial'' to ``$\d$-polynomial'') and set $K\{Y\}^{\neq} \coloneqq K\{Y\} \setminus \{0\}$.
Let $P$ range over $K\{Y\}^{\neq}$.
The \emph{order} of $P$ is the smallest $r$ such that $P \in K[Y, Y', \dots, Y^{(r)}]$.
Its \emph{degree} is its total degree.
If $r$ is the order of $P$, $m$ its degree in $Y^{(r)}$, and $n$ its (total) degree, then the \emph{complexity} of $P$ is the triple $c(P) \coloneqq (r, m, n)$;
complexities are ordered lexicographically.
For $\bm i = (i_0, \dots, i_r) \in \N^{1+r}$, we set $Y^{\bm i} \coloneqq Y^{i_0} (Y')^{i_1} \dots (Y^{(r)})^{i_r}$.
If $P$ has order at most $r$, then we decompose $P$ as $\sum_{\bm i} P_{\bm i} Y^{\bm i}$, where $\bm i$ ranges over $\N^{1+r}$.
We also sometimes decompose $P$ into its homogeneous parts, so let $P_d$ denote the homogeneous part of $P$ of degree $d$ and set $P_{\les d} \coloneqq \sum_{i \les d} P_i$ and $P_{>d} \coloneqq \sum_{i>d} P_i$.
Letting $|\bm i| \coloneqq i_0 + \dots + i_r$, we note that $P_d = \sum_{|\bm i| = d} P_{\bm i} Y^{\bm i}$.
The \emph{multiplicity of $P$ at $0$}, denoted by $\mul P$, is the least $d$ with $P_d \neq 0$.
We often use, for $a \in K$, the additive and multiplicative conjugates of $P$ by $a$ defined by $P_{+a}(Y) \coloneqq P(a+Y)$ and $P_{\x a}(Y) \coloneqq P(aY)$.
For convenience, we also write $P_{-a}$ for $P_{+(-a)}$.
Note that $(P_{+a})_{+b}=(P_{+b})_{+a}=P_{+(a+b)}$ for $b \in K$, which we write $P_{+a+b}$.
We define $P_{+a-b}$ likewise.
The \emph{multiplicity of $P$ at $a$} is $\mul P_{+a}$.
Note that $(P_d)_{\x a} = (P_{\x a})_d$, which we denote by $P_{d, \x a}$.
For more on such conjugation, see \cite[\S4.3]{adamtt}.
We extend the derivation of $K$ to $K\{Y\}$ in the natural way, and we also extend $v$ to $K\{Y\}$ by setting $v(P)$ to be the minimum valuation of the coefficients of $P$.
The relations $\prece$, $\prec$, $\asymp$, and $\sim$ are also extended to $K\{Y\}$ in the obvious way.
We recall how $v(P)$ behaves under additive and multiplicative conjugation of $P$ in \S\ref{prelim:arch}.

There are two conditions we sometimes impose connecting the valuation and the derivation.
First, we say $K$ is \emph{asymptotic} if, for all $f$, $g \in \cao^{\neq}$,
\[
f \prec g\ \iff\ f' \prec g'.
\]
If $K$ is asymptotic, then $\der$ is continuous with respect to the valuation topology on $K$ \cite[Corollary~9.1.5]{adamtt}.
It follows immediately from the definition that if $K$ is asymptotic, then $v(C^{\x}) = \{0\}$, in which case we say that $K$ has \emph{few constants}.
This weaker condition is assumed in some lemmas, so as to delay the assumption that $K$ is asymptotic until \S\ref{sec:redcomplex}.
The class of asymptotic fields includes for example the ordered (valued) differential field of logarithmic-exponential transseries studied in \cite{adamtt}, and, more generally, the class of differential-valued fields introduced by Rosenlicht \cite{rosendval}.

The second condition is more basic and is assumed throughout:
We say that $K$ has \emph{small derivation} if $\der \cao \subseteq \cao$.
Small derivation also implies that $\der$ is continuous with respect to the valuation topology on $K$; in fact, $\der$ is continuous if and only if $\der\cao \subseteq a\cao$ for some $a \in K^{\x}$ \cite[Lemma~4.4.7]{adamtt}.
It also implies that $\der \ca O \subseteq \ca O$ \cite[Lemma~4.4.2]{adamtt}, so $\der$ induces a derivation on $\bm k$.
Then we say that $\bm k$ is \emph{$r$-linearly surjective} if, for all $a_0, \dots, a_r \in \bm k$ with $a_i \neq 0$ for some $i \les r$, the equation $1 + a_0 y + a_1y' + \dots + a_r y^{(r)} = 0$ has a solution in $\bm k$.
We call $\bm k$ \emph{linearly surjective} if $\bm k$ is $r$-linearly surjective for each $r$.
Generalizing the notion of henselianity for valued fields, we say that $K$ is \emph{$r$-differential-henselian} (\emph{$r$-$\d$-henselian} for short) if:
\begin{enumerate}
	\item[($r$DH1)] $\bm k$ is $r$-linearly surjective, and 
	\item[($r$DH2)] whenever $P \in \ca O\{Y\}$ has order at most $r$ and satisfies $P_0 \prec 1$ and $P_1 \asymp 1$, there is $y \in \cao$ with $P(y) = 0$.
\end{enumerate}
The following equivalent formulation is often used without comment.
Its proof uses the Equalizer Theorem (see Theorem~\ref{adh6.0.1}), a key result from \cite{adamtt} which also underlies the results of this paper.
\begin{lem}[{\cite[7.1.1, 7.2.1]{adamtt}}]\label{adh7.1.1}
We have that $K$ is $r$-$\d$-henselian if and only if, for every $P \in \ca O\{Y\}$ of order at most $r$ satisfying $P_1 \asymp 1$ and $P_i \prec 1$ for all $i \ges 2$, there is $y \in \ca O$ with $P(y)=0$.
\end{lem}
We say that $K$ is \emph{differential-henselian} (\emph{$\d$-henselian} for short) if it is $r$-$\d$-henselian for each $r$.
These definitions are due to Aschenbrenner, van den Dries, and van der Hoeven \cite[Chapter~7]{adamtt}, although earlier notions were considered by Scanlon for monotone fields \cite{scanlon} and F.-V.\ Kuhlmann for differential-valued fields \cite{fvkultrametric}.
Connecting this to asymptoticity, we note that if $K$ is $1$-$\d$-henselian and has few constants, then it is asymptotic \cite[Lemmas~7.1.8 and 9.1.1]{adamtt}.

Throughout, by an \emph{extension of $K$} we mean a valued differential field extension of $K$ with small derivation; similarly, \emph{embedding} means ``valued differential field embedding.''
In analogy with the notion of a henselization of a valued field, we defined in \cite{dhfinrank} the notion of a differential-henselization of a valued differential field:
We call an extension $L$ of $K$ a \emph{differential-henselization} (\emph{$\d$-henselization} for short) of $K$ if it is an immediate $\d$-henselian extension of $K$ that embeds over $K$ into any immediate $\d$-henselian extension of $K$.
Finally, we call $K$ \emph{maximal} if it has no proper immediate extension and \emph{differential-algebraically maximal} (\emph{$\d$-algebraically maximal} for short) if it has no proper differentially algebraic immediate extension.
Recall from \cite[Chapter~7]{adamtt} that if the derivation induced on $\bm k$ is nontrivial, then $K$ is $\d$-algebraically maximal if and only if every pc-sequence in $K$ of $\d$-algebraic type over $K$ has a pseudolimit in $K$ (see \S\ref{prelim:imext} for this notion).
It is also worth pointing out that by the main result of \cite{maximext}, $K$ is maximal in this sense if and only if $K$ is maximal as a valued field in the usual sense, which in turn is equivalent to every pc-sequence in $K$ having a pseudolimit in $K$.

\subsection{Outline}

The main technical tool of the paper is Proposition~\ref{mainlemmadiv}, and assuming this we prove Theorems~\ref{maximmed}, \ref{dhdalgmax}, and \ref{dhensel} in \S\ref{mainresults}.
We also use Proposition~\ref{mainlemmadiv} to obtain analogues of these results relativized to $\d$-polynomials of a fixed order in \S\ref{sec:main:add}.
The rest of the paper after \S\ref{mainresults} is devoted to proving Proposition~\ref{mainlemmadiv}, and the strategy closely follows the approach taken to prove \cite[Proposition~14.5.1]{adamtt}, which is an analogue of Proposition~\ref{mainlemmadiv} in the setting of $\upomega$-free $H$-asymptotic differential-valued fields.

First, we adapt the differential newton diagram method of \cite[\S13.5]{adamtt} to the setting of valued differential fields with small derivation and divisible value group in \S\ref{sec:ndiag}, which relies in an essential way on the Equalizer Theorem \cite[Theorem~6.0.1]{adamtt}; this is where divisibility is used.
The main results are Proposition~\ref{ndiag} and Corollary~\ref{ndiagcor}, which are then connected to pc-sequences in \S\ref{sec:ndiag:ddegcut}.

From there, we proceed to study asymptotic differential equations in \S\ref{sec:ade}, with the main technical notion being that of an unraveller, adapted from \cite[\S13.8]{adamtt}.
There are three key steps in this section.
First, we establish the existence of an unraveller that is a pseudolimit of a pseudocauchy sequence in Lemma~\ref{plimunravellersexist}, via Proposition~\ref{unravellersexist}.
Second, we reduce the degree of an asymptotic differential equation in Lemma~\ref{reducedegunravel}.
Third, we find a solution of an asymptotic differential equation in a $\d$-henselian field that approximates an element in an extension of that field in \S\ref{sec:ade:dhsolns}.

The penultimate section, \S\ref{sec:redcomplex}, based on \cite[\S14.4]{adamtt}, is quite technical.
It combines many results from the previous sections and culminates in Proposition~\ref{maintechprop} and its Corollary~\ref{maintechcor}, which is essential to the proof of Proposition~\ref{mainlemmadiv}.
One of the main steps here is Lemma~\ref{slowdownlem}, which allows us to use Lemma~\ref{reducedegunravel} to reduce the degree of an asymptotic differential equation.

There are four salient differences from the approach in \cite{adamtt}.
The first is that the ``dominant part'' and ``dominant degree'' of $\d$-polynomials replace their more technical cousins ``newton polynomial'' and ``newton degree,'' leading to the simplification of some proofs.
The second is that since $K$ is not assumed to be $H$-asymptotic, we replace the convex valuation on $\Gamma$ given by $v(g) \mapsto v(g'/g)$, for $g \in K^{\x}$ with $g \not\asymp 1$, with that given by considering archimedean classes.
Third, under the assumption of $\upomega$-freeness, newton polynomials are in $C[Y](Y')^{n}$ for some $n$, but dominant parts need not have this special form.
This leads to changes in \S\ref{sec:redcomplex}, such as the need to take partial derivatives with respect to higher order derivatives of $Y$ than just $Y'$.
Finally, Lemma~\ref{mdpac} enables the generalization of Proposition~\ref{mainlemmadiv} to nondivisible value group in Proposition~\ref{mainlemmah}.
(An analogous lemma was used in \cite{newtondiv} to remove the assumption of divisible value group from the corresponding results in \cite[Chapter~14]{adamtt}.)

\subsection{Review of assumptions}
Throughout, $K$ is a valued differential field with nontrivial (surjective) valuation $v \colon K^{\x} \to \Gamma$ and nontrivial derivation $\der \colon K \to K$.
The valuation ring is $\ca O$ with maximal ideal $\cao$, and we further assume that $K$ has small derivation, i.e., $\der \cao \subseteq \cao$.
Then the differential residue field is $\bm k = \ca O/\cao$.
The field of constants is denoted by $C$.
Additional assumptions on $K$, $\Gamma$, and $\bm k$ are indicated as needed.
``Extension'' is short for ``valued differential field extension with small derivation'' and ``embedding'' is short for ``valued differential field embedding.''
We let $d$, $i$, $m$, $n$, and $r$ range over $\N = \{0, 1, 2, \dots\}$.

\section{Preliminaries}\label{prelim}

Throughout this section, $P \in K\{Y\}^{\neq}$.

\subsection{Dominant parts of differential polynomials}\label{prelim:ddeg}
We present in this subsection the notion of the dominant part of a $\d$-polynomial over $K$ when $K$ has a monomial group $\f M$, i.e., a subgroup of $K^{\x}$ that is mapped bijectively onto $\Gamma$ by $v$.
This assumption yields slightly improved versions of lemmas from \cite[\S6.6]{adamtt}, where this notion is developed without the monomial group assumption.
All the statements about dominant degree and dominant multiplicity given here remain true in that greater generality, and we freely use them later even when $K$ may not have a monomial group.
The proofs are essentially the same, so are omitted.

\begin{ass}
In this subsection, $K$ has a monomial group $\f M$.
\end{ass}

We let $\f d_P \in \f M$ be the unique monomial such that $\f d_P \asymp P$.
For $Q = 0 \in K\{Y\}$, we set $\f d_Q \coloneqq 0$.
\begin{defn}
Since $\f d_P^{-1}P \in \ca O\{Y\}$, we define the \emph{dominant part} of $P$ to be the $\d$-polynomial
\[
D_P\ \coloneqq\ \overline{\f d_P^{-1}P}\ =\ \sum_{\bm i} (\overline{P_{\bm i}/\f d_P}) \, Y^{\bm i}\ \in\ \bm k\{Y\}^{\neq}.
\]
For $Q=0 \in K\{Y\}$, we set $D_Q \coloneqq 0 \in \bm k\{Y\}$.

Then $\deg D_P \les \deg P$ and $\ord D_P \les \ord P$.
We call $\ddeg P \coloneqq \deg D_P$ the \emph{dominant degree} of $P$ and $\dmul P \coloneqq \mul D_P$ the \emph{dominant multiplicity of $P$ at $0$}.
\end{defn}
Note that if $P$ is homogeneous of degree $d$, then so is $D_P$.

\begin{lem}\label{adh6.6.2}
Let $Q \in K\{Y\}$. Then
\begin{enumerate}
	\item if $P \succ Q$, then $D_{P+Q}=D_P$;
	\item if $P \asymp Q$ and $P+Q \asymp P$, then $D_{P+Q} = D_P + D_Q$;
	\item $D_{PQ} = D_P D_Q$.
\end{enumerate}
\end{lem}
\begin{proof}
Part (ii) is not in \cite{adamtt}, so we give a proof.
Suppose $P \asymp Q$ and $P+Q \asymp P$, so $\f d_{P+Q} = \f d_P = \f d_Q$.
Then
\[
D_{P+Q}\
=\ \sum_{\bm i} \overline{ \left( P + Q \right)_{\bm i}/\f d_{P+Q}} \, Y^{\bm i}\
=\ \sum_{\bm i}\left( \overline{P_{\bm i}/\f d_P} + \overline{Q_{\bm i}/\f d_Q}\right) Y^{\bm i}\
=\ D_P + D_Q.
\qedhere
\]
\end{proof}

%\begin{lem}\label{adh6.6.3}
%Let $Q \in \ca O\{Y\}$ be such that $\overline{Q} \notin \bm k$. Then
%\[
%P(Q) \asymp P, \hs D_{P(Q)}=D_P(\overline{Q}).
%\]
%\end{lem}
%%\begin{proof}
%%Applying \cite[Lemma~4.3.12]{adamtt} to $\bm k$ gives that $D_P(\overline Q) \neq 0$, because $D_P \neq 0$ and $\overline{Q} \notin \bm k$.
%%Thus
%%\[
%%\overline{\f d_P^{-1}P(Q)}\ =\ \sum_{\bm i}(\overline{P_{\bm i}/\f d_P}) \overline{Q}^{\bm i}\ =\ D_P(\overline Q)\ \neq\ 0.
%%\]
%%In particular, $P(Q) \asymp P$, so $\f d_{P(Q)}=\f d_P$, and hence $D_{P(Q)}=D_P(\overline Q)$.
%%\end{proof}
%
%We use this to derive the following.

\begin{lem}\label{adh6.6.5}
Let $a \in K$ with $a \prece 1$. Then:
\begin{enumerate}
\item $D_{P_{+a}}=(D_P)_{+\bar{a}}$, and thus $\ddeg P_{+a}=\ddeg P$;
\item if $a \asymp 1$, then $D_{P_{\x a}}=(D_P)_{\x \bar{a}}$, $\dmul P_{\x a}=\dmul P$, and $\ddeg P_{\x a}=\ddeg P$;
\item if $a \prec 1$, then $\ddeg P_{\x a} \les \dmul P$.
\end{enumerate}
\end{lem}
%\begin{proof}
%For (i), apply Lemma~\ref{adh6.6.3} with $Q=a+Y$.
%For (ii), apply Lemma~\ref{adh6.6.3} with $Q=aY$, and recall that $(P_d)_{\x h}=(P_{\x h})_d$ for any $h \in K$.
%For (iii), the statement is trivial if $a=0$, so suppose $0\neq a \prec 1$.
%For convenience, assume that $P \asymp 1$, and set $d \coloneqq \dmul P$.
%Then
%\[
%P\ =\ \sum_{i<d}P_i+P_d+\sum_{i>d}P_i,
%\]
%where $P_i \prec 1$ for $i<d$, $P_d \asymp 1$, and $P_i \prece 1$ for $i>d$, and
%\[
%P_{\x a}\ =\ \sum_{i<d}(P_i)_{\x a}+(P_d)_{\x a}+\sum_{i>d}(P_i)_{\x a}.
%\]
%Then for $i>d$, by Lemma~\ref{adh6.1.3}, 
%\[
%v\big((P_d)_{\x a}\big)\ =\ dva+o(va)\ <\ v(P_i)+iva+o(va)\ =\ v\big((P_i)_{\x a}\big),
%\]
%from which (iii) follows.
%\end{proof}

It follows from (ii) that $\ddeg P_{\x g}$ and $\dmul P_{\x g}$ only depend on $vg$, for $g \in K^{\x}$.
%\begin{cor}\label{adh6.6.6}
%If $f-g \prece h$, then $\ddeg P_{+f, \x h}= \ddeg P_{+g, \x h}$.
%If moreover $f-g \prec h$, then $D_{P_{+f, \x h}}=D_{P_{+g, \x h}}$.
%\end{cor}
%\begin{proof}
%The second claim is not in \cite{adamtt}, so we give a proof.
%Set $\varepsilon \coloneqq (f-g)/h \prec 1$.
%Then
%\[
%D_{P_{+f, \x h}}\ =\ D_{P_{+g, \x h, +\varepsilon}}\ =\ \left(D_{P_{+g, \x h}}\right)_{+\bar\varepsilon}\ =\ D_{P_{+g, \x h}}.
%\]
%\end{proof}
The next five results are exactly as in \cite[\S6.6]{adamtt}, but are recalled here for the reader's convenience.

\begin{cor}[{\cite[6.6.6]{adamtt}}]\label{adh6.6.6}
If $f, g \in K$ and $h \in K^{\x}$ satisfy $f-g \prece h$, then
\[
\ddeg P_{+f, \x h}\ =\ \ddeg P_{+g, \x h}.
\]
\end{cor}

\begin{cor}[{\cite[6.6.7]{adamtt}}]\label{adh6.6.7}
Let $f, g \in K^{\x}$.
Then $\mul P = \mul(P_{\x f}) \les \ddeg P_{\x f}$ and
\[
f \prec g\ \implies\ \dmul P_{\x f}\ \les\ \ddeg P_{\x f}\ \les\ \dmul P_{\x g}\ \les\ \ddeg P_{\x g}.
\]
\end{cor}

Below, we let $\ca E \subseteq K^{\x}$ be nonempty and such that for $f, g \in K^{\x}$, $f \prece g \in \ca E$ implies $f \in \ca E$.
In this case, we say that $\ca E$ is \emph{$\prece$-closed},
and we consider the \emph{dominant degree of $P$ on $\ca E$} defined by
\[
\ddeg_{\ca E} P\ \coloneqq\ \max\{ \ddeg P_{\x f} : f \in \ca E \}.
\]
Note that $\prece$-closed sets correspond to nonempty upward-closed subsets of $\Gamma$.
If for $\gamma \in \Gamma$ we have $\ca E = \{ f \in K^{\x} : vf \ges \gamma \}$, then $\ddeg_{\ges \gamma} P \coloneqq \ddeg_{\ca E} P$.
For any $g\in K^\x$ with $vg=\gamma$ we set $\ddeg_{\prece g} P \coloneqq \ddeg_{\ges \gamma}P$, and by the previous result we have
$\ddeg_{\prece g} P = \ddeg P_{\x g}$.
We define $\ddeg_{>\gamma} P$ and $\ddeg_{\prec g} P$ analogously.

\begin{lem}[{\cite[6.6.9]{adamtt}}]\label{adh6.6.9}
If $v(\ca E)$ has no smallest element,
then
\[
\ddeg_{\ca E} P\ =\ \max\{ \dmul(P_{\x f}) : f \in \ca E \}.
\]
\end{lem}

\begin{lem}[{\cite[6.6.10]{adamtt}}]\label{adh6.6.10}
If $f \in \ca E$, then $\ddeg_{\ca E} P_{+f} = \ddeg_{\ca E} P$.
\end{lem}

\begin{cor}[{\cite[6.6.11]{adamtt}}]\label{adh6.6.11}
Suppose that $\ddeg_{\ca E} P=1$ and $y \in \ca E$ satisfies $P(y)=0$. If $f \in \ca E$, then
\[
\mul P_{+y, \x f}\ =\ \dmul P_{+y, \x f}\ =\ \ddeg P_{+y, \x f}\ =\ 1.
\]
\end{cor}

%\begin{cor}[{\cite[6.6.12]{adamtt}}]\label{adh6.6.12}
%If $a, b \in K$ and $\alpha, \beta \in \Gamma$ satisfy $v(b-a) \ges \alpha$ and $\beta \ges \alpha$,
%then
%\[
%\ddeg_{\ges \beta} P_{+b}\ \les\ \ddeg_{\ges \alpha} P_{+a}.
%\]
%\end{cor}

\subsection{Dominant degree in a cut}\label{ddegcut}
We recall the notion of ``dominant degree in a cut'' from \cite{dhfinrank} and some basic properties proved there.
First, for later use we mention that the condition
$\ddeg P \ges 1$ is necessary for the existence of a zero $f\prece 1$ of $P$ in an extension of $K$: 

\begin{lem}[{\cite[2.1]{dhfinrank}}]\label{ddegroot}
Let $g \in K^\x$ and suppose that $P(f)=0$ for some $f \prece g$ in some extension of $K$. Then $\ddeg P_{\x g} \ges 1$.
\end{lem}
%\begin{proof}
%Let $L$ be an extension of $K$ and suppose $f \in L$, $f \preccurlyeq g$, and $P(f)=0$. Then $f=ag$ for some $a \in L$ with $a \preccurlyeq 1$.
%Letting $Q=P_{\x g}$, we have $Q(a)=0$, so $D_Q(\bar{a})=0$. Hence, $\ddeg P_{\x g} \ges 1$.
%\end{proof}

In the rest of this section,
$(a_\rho)$ is a pc-sequence in $K$ with $\gamma_\rho \coloneqq v(a_{\rho+1}-a_\rho)$; here and later $\rho+1$ denotes the immediate successor of $\rho$ in the well-ordered set of indices.

\begin{lem}[{\cite[2.2]{dhfinrank}}]
There is an index $\rho_0$ and a number $d\big(P, (a_\rho)\big) \in \N$ such that for all $\rho>\rho_0$,
\[
\ddeg_{\ges \gamma_\rho} P_{+a_\rho}\ =\ d\big(P, (a_\rho)\big).
\]
Whenever $(b_\sigma)$ is a pc-sequence in $K$ equivalent to $(a_\rho)$, we have $d\big(P,(a_\rho)\big)=d\big(P,(b_\sigma)\big)$.
\end{lem}
%\begin{proof}
%Take $\rho_0$ such that for all $\rho'>\rho\ges\rho_0$, we have $\gamma_{\rho'}>\gamma_\rho$ and $\gamma_\rho \in \Gamma$.
%Then
%\[
%\ddeg_{\ges\gamma_{\rho'}} P_{+a_{\rho'}}\ \les\ \ddeg_{\ges\gamma_{\rho}} P_{+a_{\rho}}\ \text{ for all }\rho'>\rho\ges\rho_0
%\]
%by Corollary~\ref{adh6.6.12}.
%This gives the existence of $d\big(P,(a_\rho)\big)$. Set $d=d\big(P,(a_\rho)\big)$.
%We can assume $\rho_0$ to be so large that $\ddeg_{\ges\gamma_\rho}P_{+a_\rho}=d$ for all $\rho\ges\rho_0$.
%Let $(b_{\sigma})$ be a pc-sequence in $K$ equivalent to $(a_\rho)$, and set $\beta_\sigma=v(b_{\sigma+1}-b_\sigma)$.
%Take an index $\sigma_0$ and $e \in \N$ so that  $\beta_\sigma \in \Gamma$
%and $\ddeg_{\ges \beta_\sigma} P_{+b_\sigma}=e$ for all $\sigma \ges \sigma_0$.
%By \cite[Lemma 2.2.17]{adamtt}, we can further arrange that $b_\sigma-a_\rho \prec a_\rho-a_{\rho_0}$ and $\beta_\sigma \ges \gamma_{\rho_0}$ for all $\rho>\rho_0$ and $\sigma>\sigma_0$.
%Then for $\sigma>\sigma_0$ we have $v(b_\sigma-a_{\rho_0})=\gamma_{\rho_0}$, and so
%\[
%e\ =\ \ddeg_{\ges \beta_\sigma} P_{+b_\sigma}\ \les\ \ddeg_{\ges \gamma_{\rho_0}} P_{+a_{\rho_0}}\ =\ d,
%\]
%by Corollary~\ref{adh6.6.12}.
%By symmetry, we also have $d \les e$, so $d=e$.
%\end{proof}

We associate to each pc-sequence $(a_\rho)$ in $K$ its \emph{cut in $K$}, denoted by $c_K(a_\rho)$, such that if $(b_\sigma)$ is a pc-sequence in $K$, then
\[
c_K(a_\rho) = c_K(b_\sigma)\ \iff\ (b_\sigma)\ \text{is equivalent to}\ (a_\rho).
\]
Throughout, $\bm a \coloneqq c_K(a_\rho)$ and if $L$ is an extension of $K$, then $\bm a_L \coloneqq c_L(a_\rho)$.
Note that $c_K(a_\rho+y)$ for $y \in K$ depends only on $\bm a$ and $y$, so we let $\bm a+y$ denote $c_K(a_\rho+y)$ and $\bm a-y$ denote $c_K(a_\rho-y)$. Similarly, $c_K(a_\rho y)$ for $y \in K^\x$ depends only on $\bm a$ and $y$, so we let $\bm a\cdot y$ denote $c_K(a_\rho y)$.
\begin{defn}
The \emph{dominant degree of $P$ in the cut of $(a_\rho)$}, denoted by $\ddeg_{\bm a} P$, is the natural number $d\big(P,(a_\rho)\big)$ from the previous lemma.
\end{defn}

\begin{lem}[{\cite[2.3]{dhfinrank}}]\label{ddegbasic}
Dominant degree in a cut has the following properties:
\begin{enumerate}
\item $\ddeg_{\bm a} P \les \deg P$;
\item $\ddeg_{\bm a} P_{+y} = \ddeg_{\bm a+y} P$ for $y \in K$;
\item if $y \in K$ and $vy$ is in the width of $(a_\rho)$, then $\ddeg_{\bm a} P_{+y} = \ddeg_{\bm a} P$;
\item $\ddeg_{\bm a} P_{\x y} = \ddeg_{\bm a\cdot y} P$ for $y \in K^\x$;
\item if $Q \in K\{Y\}^{\neq}$, then $\ddeg_{\bm a} PQ=\ddeg_{\bm a}P + \ddeg_{\bm a} Q$;
\item if $P(\ell)=0$ for some pseudolimit $\ell$ of $(a_\rho)$ in an extension of $K$, then $\ddeg_{\bm a} P \ges 1$;
\item if $L$ is an extension of $K$, then $\ddeg_{\bm a} P=\ddeg_{\bm a_L} P$.
\end{enumerate}
\end{lem}
%\begin{proof}
%Items (i), (ii), (iv), and (v) follow routinely from basic facts about dominant degree found in the previous subsection, (iii) follows from (ii), and (vii) is obvious.  
%
%For (vi), let $\ell$ be a pseudolimit of $(a_\rho)$ in an extension of $K$ with $P(\ell)=0$. Take $\rho_0$ such that, for all $\rho \ges \rho_0$, $v(\ell-a_\rho)=\gamma_\rho \in \Gamma$ and 
%$d\big(P,(a_\rho)\big)=\ddeg_{\ges \gamma_\rho}P_{+a_\rho}$. Let $\rho\ges \rho_0$ and set $Q=P_{+a_\rho}$. Then $Q(\ell-a_\rho)=0$, so $\ddeg Q_{\x g} \ges 1$ for any $g \in K$ with $vg = \gamma_\rho$, by Lemma~\ref{ddegroot}.
%%For (vii), it is immediate that $\ddeg_{\bm a} P \les \ddeg_{\bm a_L} P$.
%%To get equality, let $Q=P_{+a_\rho}$ and note that $\ddeg_{\ges \gamma_\rho} Q = \max\{\ddeg Q_{\x h} : v_L(h)=\gamma_\rho\}$ by Corollary~\ref{adh6.6.7}. For any $h \in L^\x$ with $v_L(h) = \gamma_\rho$, we have $h=gu$ for some $g \in K^\x$ with $vg=\gamma_\rho$ and some $u \in L^\x$ with $v_L(u)=0$.
%%Then \cite[Lemma 6.6.5]{adamtt} gives that $\ddeg Q_{\x h}=\ddeg Q_{\x g}$.
%\end{proof}

\subsection{Constructing immediate extensions}\label{prelim:imext}
We review how to construct immediate extensions by evaluating $\d$-polynomials along pc-sequences.
%We recall three lemmas about evaluating $\d$-polynomials along pc-sequences and constructing immediate extensions.
Recall from \cite[\S4.4]{adamtt} that a pc-sequence $(a_\rho)$ in $K$ is of \emph{$\d$-algebraic type over $K$} if there is an equivalent pc-sequence $(b_\sigma)$ in $K$ and a $P \in K\{Y\}$ such that $P(b_\sigma) \pconv 0$; such a $P$ of minimal complexity is called a \emph{minimal $\d$-polynomial of $(a_\rho)$ over $K$}.
If no such $(b_\sigma)$ and $P$ exist, then $(a_\rho)$ is of \emph{$\d$-transcendental type over $K$}.
\begin{ass}
In this subsection, the derivation induced on $\bm k$ is nontrivial.
\end{ass}

\begin{lem}[{\cite[6.8.1]{adamtt}}]\label{adh6.8.1}
Let $(a_\rho)$ be a pc-sequence in $K$ with pseudolimit $a \in L$, where $L$ is an extension of $K$, and let $G \in L\{Y\} \setminus L$.
Then there exists an equivalent pc-sequence $(b_\rho)$ in $K$ such that $G(b_\rho) \pconv G(a)$.
\end{lem}

\begin{lem}[{\cite[6.9.1]{adamtt}}]\label{adh6.9.1}
Let $(a_\rho)$ be a pc-sequence in $K$ of $\d$-transcendental type over $K$.
Then $K$ has an immediate extension $K \langle a \rangle$ with $a$ $\d$-transcendental over $K$ and $a_\rho \pconv a$ such that for any extension $L$ of $K$ and any $b \in L$ with $a_\rho \pconv b$, there is a unique embedding $K\langle a \rangle \to L$ over $K$ sending $a$ to $b$.
\end{lem}

\begin{lem}[{\cite[6.9.3]{adamtt}}]\label{adh6.9.3}
Let $(a_\rho)$ be a pc-sequence in $K$ with minimal $\d$-polynomial $P$ over $K$.
Then $K$ has an immediate extension $K \langle a \rangle$ with $P(a)=0$ and $a_\rho \pconv a$ such that for any extension $L$ of $K$ and any $b \in L$ with $a_\rho \pconv b$ and $P(b)=0$, there is a unique embedding $K\langle a \rangle \to L$ over $K$ sending $a$ to $b$.
\end{lem}

\subsection{Constructing immediate extensions and vanishing}\label{prelim:imextfo}
\begin{ass}
In this subsection, the derivation induced on $\bm k$ is nontrivial and $\Gamma$ has no least positive element.
\end{ass}
The notion of minimal $\d$-polynomial is not first-order, so we include here a first-order variant of Lemma~\ref{adh6.9.3} that is a special case of \cite[Lemma~5.3]{maximext}.
We then connect it to dominant degree in a cut.
Under the assumptions above, all extensions of $K$ are \emph{strict} \cite[Lemma~1.3]{maximext}, and $K$ is \emph{flexible} \cite[Lemma~1.15 and Corollary~3.4]{maximext}.
These notions are defined in \cite{maximext} but are incidental here, and mentioned only since they occur in the corresponding lemmas of \cite{maximext}.

Let $\ell \notin K$ be an element in an extension of $K$ such that $v(\ell - K) \coloneqq \{v(\ell-x) : x \in K\}$ has no largest element (equivalently, $\ell$ is the pseudolimit of some divergent pc-sequence in $K$).
We say that $P$ \emph{vanishes at $(K,\ell)$} if for all $a \in K$ and $\f v \in K^\x$ with $a-\ell \prec \f v$, $\ddeg_{\prec \f v} P_{+a} \ges 1$.
Then $Z(K, \ell)$ denotes the set of nonzero $\d$-polynomials over $K$ vanishing at $(K,\ell)$.

%The first lemma is a special case of \cite[Lemma~4.2]{maximext}.
%\begin{lem}\label{pnotinzatf}
%Suppose $P \notin Z(K, \ell)$.
%Let $a \in K$ and $\f v \in K^{\x}$ be such that $a-\ell \prec \f v$ and $\ddeg_{\prec \f v} P_{+a} = 0$.
%Then $P(f) \sim P(a)$ for any $f$ in an extension of $K$ satisfying $f-a \prec \f v$.
%\end{lem}

%\begin{lem}\label{maximext:5.2}
%Suppose $Z(K, \ell)=\0$.
%Then $P(\ell)\neq 0$ for all $P$, and $K\langle \ell \rangle$ is an immediate extension of $K$.
%Suppose also that $M$ is an extension of $K$ and $g \in M$ satisfies $v(a-g)=v(a-\ell)$ for all $a \in K$.
%Then there is a unique embedding $K\langle \ell \rangle \to M$ over $K$ sending $\ell$ to $g$.
%\end{lem}

\begin{lem}[{\cite[5.3]{maximext}}]\label{maximext:5.3}
Suppose that $Z(K, \ell)\neq\0$ and $P \in Z(K, \ell)$ has minimal complexity.
Then $K$ has an immediate extension $K\langle f\rangle$ such that $P(f)=0$ and $v(a-f)=v(a-\ell)$ for all $a \in K$.
Moreover, if $M$ is an extension of $K$ and $s \in M$ satisfies $P(s)=0$ and $v(a-s)=v(a-\ell)$ for all $a \in K$, then there is a unique embedding $K\langle f\rangle \to M$ over $K$ sending $f$ to $s$.
\end{lem}

\begin{lem}[{\cite[4.6]{maximext}}]\label{pconv0vanish}
Suppose that $(a_\rho)$ is a divergent pc-sequence in $K$ with $a_\rho \pconv \ell$.
If $P(a_\rho) \pconv 0$, then $P \in Z(K, \ell)$.
\end{lem}

\begin{lem}[{\cite[4.7]{maximext}}]\label{ddegcutfo}
Suppose that $(a_\rho)$ is a divergent pc-sequence in $K$ with $a_\rho \pconv \ell$.
Then
\[
\ddeg_{\bm a} P\ =\ \min\{\ddeg_{\prec \f v} P_{+a} : a-\ell \prec \f v\}.
\]
In particular, $\ddeg_{\bm a}P \ges 1 \iff P \in Z(K, \ell)$.
\end{lem}

\begin{cor}\label{mindiffpolyfo}
Suppose that $(a_\rho)$ is a divergent pc-sequence in $K$ with $a_\rho \pconv \ell$.
The following conditions on $P$ are equivalent:
\begin{enumerate}
\item $P \in Z(K, \ell)$ and has minimal complexity in $Z(K, \ell)$;
\item $P$ is a minimal $\d$-polynomial of $(a_\rho)$ over $K$.
\end{enumerate}
\end{cor}
\begin{proof}
The proof is the same as that of \cite[Corollary~11.4.13]{adamtt}, using Lemma~\ref{maximext:5.3} in place of \cite[Lemma~11.4.8]{adamtt}, Lemma~\ref{adh6.8.1} in place of \cite[Lemma~11.3.8]{adamtt}, and Lemma~\ref{pconv0vanish} in place of \cite[Lemma~11.4.11]{adamtt}.
\end{proof}
In particular, $Z(K, \ell) = \0$ if and only if $(a_\rho)$ is of $\d$-transcendental type over $K$, and $Z(K, \ell) \neq \0$ if and only if $(a_\rho)$ is of $\d$-algebraic type over $K$.

\subsection{Archimedean classes and coarsening}\label{prelim:arch}

For $\gamma \in \Gamma$, we let $[\gamma]$ denote its archimedean class.
That is,
\[
[\gamma]\ \coloneqq\ \{ \delta \in \Gamma : |\delta| \les n|\gamma|\ \text{and}\ |\gamma| \les n|\delta|\ \text{for some}\ n\}.
\]
We order the set $[\Gamma] \coloneqq \{ [\gamma] : \gamma \in \Gamma\}$ by $[\delta] < [\gamma]$ if $n|\delta| < |\gamma|$ for all $n$.
Giving the set of archimedean classes their reverse order,
the map $\gamma \mapsto [\gamma]$ is a convex valuation on $\Gamma$ (see \cite[\S2.2]{adamtt} for the notion of a valuation on an abelian group and \cite[\S2.4]{adamtt} for that of a \emph{convex} valuation on an \emph{ordered} abelian group).
In particular, the implication $[\delta]<[\gamma] \implies [\delta+\gamma]=[\gamma]$ is often used.
Then for $\phi \in K^{\x}$ with $\phi \not\asymp 1$, the set $\Gamma_{\phi} \coloneqq \{\gamma : [\gamma]<[v\phi]\}$ is a convex subgroup of $\Gamma$.
We will use $v_\phi$, the coarsening of $v$ by $\Gamma_\phi$, and its corresponding dominance relation, $\prece_\phi$, defined by
\begin{align*}
v_{\phi} \colon K^{\x} &\to \Gamma/\Gamma_\phi\\
a &\mapsto va + \Gamma_\phi,
\end{align*}
and $a \prece_\phi b$ if $v_\phi(a) \ges v_\phi(b)$.
Note that the symbols $v_\phi$ and $\prece_\phi$ also appeared in \cite[\S9.4]{adamtt}, where they indicated a different coarsening of $v$.

We first recall how $v(P)$ changes as we additively and multiplicatively conjugate $P$.

\begin{lem}[{\cite[4.5.1]{adamtt}}]\label{adh4.5.1}
Let $f \in K$.
\begin{enumerate}
	\item If $f \prece 1$, then $P_{+f} \asymp P$; if $f \prec 1$, then $P_{+f} \sim P$.
	\item If $f \neq 0$, then $v(P_{\x f}) \in \Gamma$ depends only on $vf \in \Gamma$.
\end{enumerate}
\end{lem}

Item (ii) allows us to define the function $v_{P} \colon \Gamma \to \Gamma$ by $vf \mapsto v(P_{\x f})$.
The main property of this function is recorded in the following lemma.
Here, for $\alpha, \beta \in \Gamma$ we write  $\alpha = o(\beta)$ if $[\alpha]<[\beta]$.

\begin{lem}[{\cite[6.1.3, 6.1.5]{adamtt}}]\label{adh6.1.3}
Let $P, Q \in K\{Y\}^{\neq}$ be homogeneous of degrees $m$, $n$, respectively.
For $\alpha, \beta \in \Gamma$ with $\alpha \neq \beta$, we have
\[
v_P(\alpha) - v_P(\beta)\ =\ m(\alpha-\beta) + o(\alpha-\beta).
\]
It follows that, for $\gamma \in \Gamma^{\neq}$,
\[
v_P(\gamma)-v_Q(\gamma)\ =\ v(P)-v(Q) + (m-n)\gamma + o(\gamma),
\]
and if $m>n$, then $v_P-v_Q$ is strictly increasing.
\end{lem}

The lemmas from the rest of this subsection will play an important role in \S\ref{sec:redcomplex}.
Lemmas~\ref{precgprecf}--\ref{switchasymp} are variants of lemmas from the end of \cite[\S9.4]{adamtt}.
The first two are facts about valued fields, not involving the derivation.

\begin{lem}\label{precgprecf}
Let $f, g \in K^{\x}$ with $f, g \not\asymp 1$.
Then $f \prec_g g \implies f \prec_f g$.
\end{lem}
\begin{proof}
From $f/g \prec_g 1$, we obtain $[vf-vg] \ges [vg]$.
If $[vf-vg]<[vf]$, then
\[
[vf]\ =\ [vf-(vf-vg)]\ =\ [vg],
\]
contradicting $[vf-vg] \ges [vg]$.
Thus $[vf-vg] \ges [vf]$.
But since $vf-vg>0$, we have $f/g \prec_f 1$, that is, $f \prec_f g$.
\end{proof}

\begin{lem}\label{twocoarsen}
Let $\phi_1, \phi_2 \in K^{\x}$ with $\phi_1, \phi_2 \not\asymp 1$ and $[v\phi_1] \les [v\phi_2]$.
Then for all $f, g \in K$, we have
\[
f \prece_{\phi_1} g \implies f \prece_{\phi_2} g \qquad \text{and} \qquad f \prec_{\phi_2} g \implies f \prec_{\phi_1} g.
\]
In particular, $f \asymp_{\phi_1} g \implies f \asymp_{\phi_2} g$.
\end{lem}
\begin{proof}
Note that for $\phi \in K^{\x}$ with $\phi \not\asymp 1$, $f \prece_{\phi} g$ if and only if $vf-vg \in \Gamma_{\phi}$ or $vf>vg$, and $f \prec_{\phi} g$ if and only if $vf-vg > \Gamma_{\phi}$.
Both implications then follow from $\Gamma_{\phi_1} \subseteq \Gamma_{\phi_2}$.
\end{proof}

\begin{lem}\label{multconjhomog}
Suppose that $P$ is homogeneous of degree $d$ and let $g \in K^{\x}$ with $g \not\asymp 1$.
Then
\[
P_{\x g}\ \asymp_g\ g^d P.
\]
\end{lem}
\begin{proof}
By Lemma~\ref{adh6.1.3}, $v(P_{\x g})=vP+d\, vg+o(vg)$, so $v_g(P_{\x g})=v_g(g^d P)$.
\end{proof}

\begin{lem}\label{asympg}
Suppose that $g \in K^{\x}$ with $g \prec 1$ and $d=\dmul P = \mul P$.
Then $P_{\x g} \asymp_g g^d P$.
\end{lem}
\begin{proof}
Since $d=\dmul P$, we have $P_i \prece P_d$ for $i \ges d$.
Since $g \prec 1$, we also have $g \prec_g 1$.
Hence $g^i P_i \prec_g g^d P_d$ for $i>d$, so
\[
P_{\x g}\ \asymp_g\ P_{d, \x g}\ \asymp_g\ g^d P_d
\]
by Lemma~\ref{multconjhomog}.
In view of $P_d \asymp P$, this yields $P_{\x g} \asymp_g g^d P$.
\end{proof}

\begin{lem}\label{preceg}
Suppose that $g \in K^{\x}$ with $g \succ 1$ and $d=\ddeg P=\ddeg P_{\x g}$.
Then $g P_{>d} \prece_g P$.
\end{lem}
\begin{proof}
If $P_{>d}=0$, then the result holds trivially, so assume that $P_{>d} \neq 0$.
Take $i>d$ such that $P_i \asymp P_{>d}$.
Then Lemma~\ref{multconjhomog} and the fact that $g \succ 1$ give
\[
(P_{\x g})_i\ =\ (P_i)_{\x g}\ \asymp_g\ g^i P_i\ \succe\ g^{d+1} P_i\ \asymp\ g^{d+1} P_{>d},
\]
so $(P_{\x g})_{>d} \succe_g g^{d+1}P_{>d}$.
Since $\ddeg P_{\x g}=d$, we also have $(P_{\x g})_d \succ (P_{\x g})_{>d}$.
As $\ddeg P=d$, $P_d \asymp P$, and so
\[
g^d P\ \asymp\ g^d P_d\ \asymp_g\ (P_{\x g})_d\ \succ\ (P_{\x g})_{>d}\ \succe_g\ g^{d+1}P_{>d},
\]
using Lemma~\ref{multconjhomog} again.
Hence $P \succe_g g P_{>d}$.
\end{proof}

\begin{lem}\label{switchasymp}
Let $f, g \in K^\x$ with $f, g \not\asymp 1$ and $[vf]<[vg]$.
Then $P_{\x f}\asymp_g P$.
\end{lem}
\begin{proof}
Take $d$ with $P_{\x f} \asymp (P_{\x f})_{d}$, so $P_{\x f} \asymp_f f^d P_d$ by Lemma~\ref{multconjhomog}.
Then Lemma~\ref{twocoarsen} gives $P_{\x f} \asymp_g f^d P_d$.
As $[vf]<[vg]$, we get $f \asymp_g 1$, and thus $P_{\x f} \asymp_g P_d \prece P$, so $P_{\x f} \prece_g P$.
Now, apply the same argument to $P_{\x f}$ and $f^{-1}$ in place of $P$ and $f$, using that $[v(f^{-1})]=[-vf]=[vf]$, to get $P=(P_{\x f})_{\x f^{-1}} \prece_g P_{\x f}$, and hence $P \asymp_g P_{\x f}$.
\end{proof}

\begin{ass}
In the next two results, $K$ has a monomial group $\f M$.
\end{ass}
Let $\f m$, $\f n$ range over $\f M$.
These two results are based on \cite[Lemma~13.2.3 and Corollary~13.2.4]{adamtt}.

\begin{lem}
Let $\f n \neq 1$ and $[v\f m]<[v\f n]$.
Suppose that $P=Q+R$ with $R \prec_{\f n} P$.
Then
\[
D_{P_{\xf m}}\ =\ D_{Q_{\xf m}}.
\]
\end{lem}
\begin{proof}
From $R \prec_{\f n} Q$, we get $R \prec Q$, so if $\f m=1$, then $D_P=D_Q$.
Now assume that $\f m \neq 1$.
Then Lemma~\ref{switchasymp} gives
\[
R_{\xf m}\ \asymp_{\f n}\ R\ \prec_{\f n}\ Q\ \asymp_{\f n}\ Q_{\xf m},
\]
so $R_{\xf m} \prec Q_{\xf m}$, and hence $D_{P_{\xf m}}=D_{Q_{\xf m}}$.
\end{proof}

\begin{cor}\label{reducedeg}
Suppose that $\f n \succ 1$ and $\ddeg P=\ddeg P_{\xf n}=d$.
Let $Q \coloneqq P_{\les d}$.
Then for all $\f m$ with $[v\f m]<[v\f n]$ and all $g \prece 1$ in $K$, we have 
\[
D_{P_{+g, \xf m}}\ =\ D_{Q_{+g, \xf m}}.
\]
\end{cor}
\begin{proof}
Let $R \coloneqq P-Q=P_{>d}$.
Then Lemma~\ref{preceg} gives
\[
R\ \prece_{\f n}\ \f n^{-1} P\ \prec_{\f n}\ P.
\]
Let $g \prece 1$.
Then $R_{+g} \asymp R$ and $P_{+g} \asymp P$ by Lemma~\ref{adh4.5.1}(i).
Thus we have $R_{+g} \prec_{\f n} P_{+g}$, so it remains to apply the previous lemma.
\end{proof}

\section{Main results}\label{mainresults}

Assuming Proposition~\ref{mainlemmadiv}, we prove here the main results of this paper concerning the uniqueness of maximal immediate extensions, the relationship between $\d$-algebraic maximality and $\d$-henselianity, and the existence and uniqueness of $\d$-henselizations.
The proof of Proposition~\ref{mainlemmadiv} is given in \S\ref{sec:mainlemmadiv}.

\begin{prop}\label{mainlemmadiv}
Suppose that $K$ is asymptotic, $\Gamma$ is divisible, and $\bm k$ is $r$-linearly surjective.
Let $(a_\rho)$ be a pc-sequence in $K$ with minimal $\d$-polynomial $G$ over $K$ of order at most $r$.
Then $\ddeg_{\bm a} G=1$.
\end{prop}

\subsection{Removing divisibility}
In the next lemmas, we construe the algebraic closure $K^{\ac}$ of $K$ as a valued differential field extension of $K$:
the derivation of $K$ extends uniquely to $K^{\ac}$ \cite[Lemma~1.9.2]{adamtt} and we equip $K^{\ac}$ with any valuation extending that of $K$.
This determines $K^{\ac}$ as a valued differential field extension of $K$ up to isomorphism over $K$, with value group the divisible hull $\Q\Gamma$ of $\Gamma$ and residue field the algebraic closure $\bm k^{\ac}$ of $\bm k$.
If $K$ is henselian, then its valuation extends uniquely to $K^{\ac}$ (see \cite[Proposition~3.3.11]{adamtt}).
By \cite[Proposition~6.2.1]{adamtt}, $K^{\ac}$ has small derivation.

\begin{lem}\label{mdpac}
Suppose that $K$ is henselian and the derivation on $\bm k$ is nontrivial.
Let $(a_\rho)$ be a pc-sequence in $K$ with minimal $\d$-polynomial $P$ over $K$.
Then $P$ remains a minimal $\d$-polynomial of $(a_\rho)$ over the algebraic closure $K^{\ac}$ of $K$.
\end{lem}
\begin{proof}
We may suppose that $(a_\rho)$ is divergent in $K$, the other case being trivial.
Then $(a_\rho)$ is still divergent in $K^{\ac}$:
If it had a pseudolimit $a \in K^{\ac}$, then we would have $Q(a_\rho) \pconv 0$, where $Q \in K[Y]$ is the minimum polynomial of $a$ over $K$ (see \cite[Proposition~3.2.1]{adamtt}).
But since $K$ is henselian, it is algebraically maximal (see \cite[Corollary~3.3.21]{adamtt}), and then $(a_\rho)$ would have a pseudolimit in $K$.

Now suppose to the contrary that $Q$ is a minimal $\d$-polynomial of $(a_\rho)$ over $K^{\ac}$ with $c(Q)<c(P)$. 
Take an extension $L \subseteq K^{\ac}$ of $K$ with $[L:K]=n$ and $Q \in L\{Y\}$.
Then as $K$ is henselian,
\[
[L:K]\ =\ [\Gamma_L:\Gamma] \cdot [\bm k_L : \bm k]
\]
(see \cite[Corollary 3.3.49]{adamtt}),
so we have a valuation basis $\ca B=\{e_1, \dots, e_n\}$ of $L$ over $K$ (see \cite[Proposition 3.1.7]{adamtt}).
That is, $\ca B$ is a basis of $L$ over $K$, and for all $a_1,\dots,a_n \in K$,
\[
v \left( \sum_{i=1}^n a_i e_i \right)\ =\ \min_{1 \les i \les n} v(a_i e_i).
\]
Then by expressing the coefficients of $Q$ in terms of the valuation basis,
\[
Q(Y)\ =\ \sum_{i=1}^n R_i(Y) \cdot e_i,
\]
where $R_i \in K\{Y\}$ for $1 \les i \les n$.

Since $Q$ is a minimal $\d$-polynomial of $(a_\rho)$ over $K^{\ac}$, by Lemma~\ref{adh6.9.3} we have an immediate extension $K^{\ac}\langle a\rangle$ of $K^{\ac}$ with $a_\rho \pconv a$ and $Q(a)=0$.
Then by Lemma~\ref{adh6.8.1}, we have a pc-sequence $(b_\rho)$ in $K$ equivalent to $(a_\rho)$ such that $Q(b_\rho) \pconv Q(a)=0$.
Finally, after passing to a cofinal subsequence, we can assume that we have $i$ with $Q(b_\rho) \asymp R_i(b_\rho) \cdot e_i$ for all $\rho$.
Then $R_i(b_\rho) \pconv 0$ and $c(R_i)<c(P)$, contradicting the minimality of $P$.
\end{proof}

Since minimal $\d$-polynomials are irreducible, note that a corollary of this lemma is that minimal $\d$-polynomials over henselian $K$ (with nontrivial induced derivation on $\bm k$) are absolutely irreducible.
We can now replace the divisibility assumption in the main proposition with that of henselianity.

\begin{prop}\label{mainlemmah}
Suppose that $K$ is asymptotic and henselian, and that $\bm k$ is linearly surjective.
Let $(a_\rho)$ be a pc-sequence in $K$ with minimal $\d$-polynomial $G$ over $K$.
Then $\ddeg_{\bm a} G=1$.
\end{prop}
\begin{proof}
By the previous lemma, $G$ remains a minimal $\d$-polynomial of $(a_\rho)$ over $K^{\ac}$.
Note that the value group of $K^{\ac}$ is divisible, and its differential residue field is linearly surjective, as an algebraic extension of $\bm k$ \cite[Corollary~5.4.3]{adamtt}.
But then $\ddeg_{\bm a_{K^{\ac}}} G = 1$ by Proposition~\ref{mainlemmadiv}, and hence $\ddeg_{\bm a} G=1$ by Lemma~\ref{ddegbasic}(vii).
\end{proof}

\subsection{Main results}

Recall from \cite{dhfinrank} that $K$ has the \emph{differential-henselian configuration property} (\emph{dh-configuration property} for short) if for every divergent pc-sequence $(a_\rho)$ in $K$ with minimal $\d$-polynomial $G$ over $K$, we have $\ddeg_{\bm a} G = 1$.
In that paper, van den Dries and the present author showed that several results follow from the dh-configuration property.
In particular, Lemma~\ref{dalgmaxplimroot}, Theorem~\ref{dhdalgmax}, and Corollary~\ref{mindhensel} follow immediately in view of Proposition~\ref{mainlemmah}.
For readability, we include their proofs from that paper verbatim.
Theorems~\ref{maximmed} and \ref{dhensel} require minor modifications involving henselizations.

The next lemma has the same proof as \cite[Lemma~4.1]{dhfinrank}, except for using Proposition~\ref{mainlemmah}.

\begin{lem}\label{dalgmaxplimroot}
Suppose $K$ is asymptotic and henselian, and $\bm k$ is linearly surjective.
Let $(a_\rho)$ be a pc-sequence in $K$ with minimal $\d$-polynomial $G$ over $K$.
Let $L$ be a $\d$-algebraically maximal extension of $K$ such that $\bm k_L$ is linearly surjective.
Then there is $b \in L$ with $a_\rho \pconv b$ and $G(b)=0$.
\end{lem}
\begin{proof}
Note that $L$ is $\d$-henselian by \cite[Theorem~7.0.1]{adamtt}.
Since $L$ is $\d$-algebraically maximal and the derivation of $\bm k_L$ is nontrivial, every pc-sequence in $L$ of $\d$-algebraic type over $L$ has a pseudolimit in $L$ by Lemma~\ref{adh6.9.3}.
Thus we  get $a \in L$ with $a_\rho \pconv a$. Passing to an equivalent pc-sequence we arrange that $G(a_\rho) \pconv 0$.
With $\gamma_\rho =v(a_{\rho+1}-a_{\rho})=v(a-a_\rho)$, eventually, Proposition~\ref{mainlemmah} gives $g_\rho \in K$ with $v(g_\rho)=\gamma_\rho$ and $\ddeg G_{+a_\rho, \times g_\rho}=1$, eventually.
By Corollary~\ref{adh6.6.6}, $\ddeg G_{+a, \times g_\rho}=1$, eventually.
We have $G(a+Y)=G(a)+A(Y)+R(Y)$ where $A$ is linear and homogeneous and all monomials in $R$ have degree $\ges 2$, and so
\[
G_{+a, \times g_{\rho}}(Y)\ =\ G(a)+A_{\times g_{\rho}}(Y) + R_{\times g_{\rho}}(Y).
\]
Now $\ddeg G_{+a,\times g_{\rho}}=1$ eventually, so
$v(G(a)) \ges v_A(\gamma_\rho)<v_R(\gamma_\rho)$ eventually.
With $v(G,a)$ as in \cite[\S7.3]{adamtt} we get $v(G,a)>v(a-a_\rho)$, eventually.
Then \cite[Lemma~7.3.1]{adamtt} gives $b \in L$ with $v_L(a-b)=v(G,a)$ and $G(b)=0$, so $v_L(a-b)>v(a-a_\rho)$ eventually. Thus $a_\rho \pconv b$.
\end{proof}

For the next result, we copy the proof of \cite[Theorem~4.2]{dhfinrank}, except for an argument involving henselizations.

\begin{thm}\label{maximmed}
Suppose that $K$ is asymptotic and $\bm k$ is linearly surjective.
Then any two maximal immediate extensions of $K$ are isomorphic over $K$.
Also, any two $\d$-algebraically maximal $\d$-algebraic immediate extensions of $K$ are isomorphic over $K$.
\end{thm}
\begin{proof}
Let $L_0$ and $L_1$ be maximal immediate extensions of $K$.
By Zorn's Lemma we have a maximal isomorphism $\mu \colon F_0 \to F_1$ over $K$ between valued differential subfields $F_i\supseteq K$ of $L_i$ for $i=0,1$, where ``maximal'' means that $\mu$ does not extend to
an isomorphism between strictly larger such valued differential
subfields. 
First, $F_i$ is asymptotic by \cite[Lemmas~9.4.2 and 9.4.5]{adamtt}, and $\bm k_{F_i}$ is linearly surjective, as $F_i$ is an immediate extension of $K$ for $i=0,1$.
Next, $F_i$ must be henselian, since its henselization in $L_i$ is algebraic over $F_i$, and thus a valued differential subfield of $L_i$ for $i=0,1$.
Now suppose towards a contradiction that $F_0 \neq L_0$ (equivalently, $F_1 \neq L_1$).
Then $F_0$ is not spherically complete, so we have a divergent pc-sequence $(a_\rho)$ in $F_0$.

Suppose that $(a_\rho)$ is of $\d$-transcendental type over $F_0$.
The spherical completeness of $L_0$ and $L_1$ then
yields $f_0 \in L_0$ and $f_1 \in L_1$
such that $a_\rho \pconv f_0$ and $\mu(a_\rho) \pconv f_1$. 
Hence by Lemma~\ref{adh6.9.1} we obtain an isomorphism $F_0\langle f_0 \rangle \to F_1\langle f_1 \rangle$ extending $\mu$, contradicting the maximality of $\mu$.

Suppose that $(a_\rho)$ is of $\d$-algebraic type over $F_0$, with minimal $\d$-polynomial $G$ over $F_0$.
Then Lemma~\ref{dalgmaxplimroot} gives $f_0 \in L_0$ with $a_\rho \pconv f_0$ and $G(f_0)=0$, and $f_1 \in L_1$ with $\mu(a_\rho) \pconv f_1$ and $\mu(G)(f_1)=0$, where we extend $\mu$ to the differential ring isomorphism $\mu \colon F_0\{Y\} \to F_1\{Y\}$ with $Y \mapsto Y$.
Now Lemma~\ref{adh6.9.3} gives an isomorphism $F_0\langle f_0 \rangle \to F_1\langle f_1 \rangle$ extending $\mu$,
and we have again a contradiction. Thus $F_0=L_0$ and
hence $F_1=L_1$ as well. 

The proof of the second statement is the same, using only Lemma~\ref{adh6.9.3}.
\end{proof}

By Proposition~\ref{mainlemmah} and \cite[Theorem~4.3]{dhfinrank}, we can now remove monotonicity from \cite[Theorem~7.0.3]{adamtt} without any further assumptions.
\begin{thm}\label{dhdalgmax}
If $K$ is asymptotic and $\d$-henselian, then it is $\d$-algebraically maximal.
\end{thm}
\begin{proof}
Let $(a_\rho)$ be a pc-sequence in $K$ with minimal $\d$-polynomial $G$ over $K$.
Towards a contradiction, assume that $(a_\rho)$ is divergent in $K$.
Then Lemma~\ref{adh6.9.3} shows that $G$ has order at least $1$ (since $K$ is henselian) and provides a proper immediate extension $K\langle a \rangle$ of $K$ with $a_\rho \pconv a$. 
Replacing $(a_\rho)$ by an equivalent pc-sequence in $K$, we arrange that $G(a_\rho) \pconv 0$.

By Proposition~\ref{mainlemmah}, $\ddeg_{\bm a} G=1$.
Taking $g_\rho \in K$ with $v(g_\rho)=\gamma_\rho$ we have
$\ddeg G_{+a_\rho, \x g_\rho} = 1$, eventually. By removing some initial terms of the sequence, we arrange that this holds for all $\rho$ and that $v(a-a_\rho)=\gamma_\rho$ for all $\rho$.
By $\d$-henselianity, we have $z_\rho \in K$ with $G(z_\rho)=0$ and $a_\rho - z_\rho \preccurlyeq g_\rho$.
From $a-a_\rho \asymp g_\rho$ we get $a-z_\rho \preccurlyeq g_\rho$.

Let $r\ges 1$ be the order of $G$.
By \cite[Lemma~2.2.19]{adamtt}, $(\gamma_\rho)$ is cofinal in $v(a-K)$, so there are indices $\rho_0<\rho_1<\dots<\rho_{r+1}$ such that $a-z_{\rho_j} \prec a-z_{\rho_i}$, for $0 \les i<j \les r+1$.
Then
\[
z_{\rho_i}-z_{\rho_{i-1}}\ \asymp\ a-z_{\rho_{i-1}}\ \succ\ a-z_{\rho_i}\ \asymp\ z_{\rho_{i+1}}-z_{\rho_i},\ \text{for}\ 1 \les i \les r.
\]
We have $a-z_{\rho_{r+1}} \prec a-z_{\rho_0} \prece g_{\rho_0}$, so $z_{\rho_0}-z_{\rho_{r+1}} \prece g_{\rho_0}$, and thus also $a_{\rho_0}-z_{\rho_{r+1}} \prece g_{\rho_0}$.
Hence
\[
\ddeg G_{+z_{\rho_{r+1}}, \x g_{\rho_0}}\ =\ \ddeg G_{+z_{\rho_0}, \x g_{\rho_0}}\ =\ \ddeg G_{+a_{\rho_0}, \times g_{\rho_0}}\ =\ 1
\]
by Corollary~\ref{adh6.6.6}.
Thus, with $z_{\rho_i}$ in the role of $y_i$, for $0 \les i \les r+1$, and $g_{\rho_0}$ in the role of $g$, we have reached a contradiction with \cite[Lemma~7.5.5]{adamtt}.
\end{proof}

The following generalizes \cite[Theorem~1.3]{dhfinrank}, removing the assumption on the value group.
Its proof is the same, except for the use of the henselization.
\begin{thm}\label{dhensel}
Suppose that $K$ is asymptotic and $\bm k$ is linearly surjective.
Then $K$ has a $\d$-henselization which has no proper differential subfield containing $K$ that is $\d$-henselian.
In particular, any two $\d$-henselizations of $K$ are isomorphic over $K$.
\end{thm}
\begin{proof}
We can assume that $K$ is henselian, as $K$ has a henselization that embeds (uniquely) over $K$ into any $\d$-henselian extension of $K$.
By \cite[Corollary~9.4.11]{adamtt} we have an immediate asymptotic $\d$-henselian extension $K^{\dh}$ of $K$ that is $\d$-algebraic over $K$ and has no proper differential subfield containing $K$ that is $\d$-henselian.
%By Theorems~\ref{dhdalgmax} and \ref{maximmed}, there is up to isomorphism over $K$ just one such extension.

Let $L$ be an immediate $\d$-henselian extension of $K$; then $L$ is asymptotic by \cite[Lemmas~9.4.2 and 9.4.5]{adamtt}.
To see that $K^{\dh}$ embeds into $L$ over $K$, use an argument similar to that in the proof of Theorem~\ref{maximmed}, as $K^{\dh}$ and $L$ are $\d$-algebraically maximal by Theorem~\ref{dhdalgmax}.
Thus $K^{\dh}$ is a $\d$-henselization of $K$ and any $\d$-henselization of $K$ is isomorphic over $K$ to $K^{\dh}$. 
\end{proof}

In fact, the argument shows that $K^{\dh}$ as in the proof of Theorem~\ref{dhensel} embeds over $K$ into any (not necessarily immediate) asymptotic $\d$-henselian extension of $K$.
This corollary has the same proof as \cite[Corollary~4.6]{dhfinrank}.
\begin{cor}\label{mindhensel}
If $K$ is asymptotic and $\bm k$ is linearly surjective, then any immediate $\d$-henselian extension of $K$ that is $\d$-algebraic over $K$ is a $\d$-henselization of~$K$.
\end{cor}
\begin{proof}
Let $K^{\dh}$ be the $\d$-henselization of $K$ from the proof of Theorem~\ref{dhensel}. Then $K^{\dh}$ is asymptotic, so is $\d$-algebraically maximal by Theorem~\ref{dhdalgmax}.
Hence any embedding $K^{\dh}\to L$ over $K$ into an immediate $\d$-algebraic extension $L$ of $K$ is surjective.
\end{proof}

\subsection{Additional results}\label{sec:main:add}

We also record versions of the above results relativized to $\d$-polynomials of a given order.
In these results, we assume that $\Gamma$ is divisible but not that $K$ is henselian.
The proofs are the same as above, except for using Proposition~\ref{mainlemmadiv} in place of Proposition~\ref{mainlemmah} and not needing to use henselizations.

To state these results, we make some definitions.
If $E$ and $F$ are differential fields, then we say that $F$ is \emph{$r$-differentially algebraic} (\emph{$r$-$\d$-algebraic} for short) over $E$ if for each $a \in F$ there are $a_1, \dots, a_n \in F$ such that $a \in E\langle a_1, \dots, a_n \rangle$ and, for $i=0, \dots, n-1$, $a_{i+1}$ is $\d$-algebraic over $E\langle a_1, \dots, a_i \rangle$ with minimal annihilator of order at most $r$.
It is routine to prove that if $L$ is $r$-$\d$-algebraic over $F$ and $F$ is $r$-$\d$-algebraic over $E$, then $L$ is $r$-$\d$-algebraic over $E$.

\begin{defn}
We call $K$ \emph{$r$-differential-algebraically maximal} (\emph{$r$-$\d$-algebraically maximal} for short) if it has no proper immediate $r$-$\d$-algebraic extension.
\end{defn}
By Zorn, $K$ has an immediate $r$-$\d$-algebraic extension that is $r$-$\d$-algebraically maximal.
Note that $K$ is $\d$-algebraically maximal if and only if it is $r$-$\d$-algebraically maximal for all $r$.
In addition, if the derivation induced on $\bm k$ is nontrivial, then by Lemmas~\ref{adh6.8.1} and \ref{adh6.9.3} $K$ is $r$-$\d$-algebraically maximal if and only if every pc-sequence in $K$ with minimal $\d$-polynomial over $K$ of order at most $r$ has a pseudolimit in $K$.

Note that $K$ being $0$-$\d$-algebraically maximal means that it has no proper immediate valued differential field extension with small derivation that is algebraic over $K$.
Since each algebraic field extension of $K$, given any valuation extending that of $K$ and the unique derivation extending that of $K$, has small derivation \cite[Proposition~6.2.1]{adamtt}, $K$ is $0$-$\d$-algebraically maximal if and only if it is algebraically maximal \emph{as a valued field}.
Thus the results below for $r=0$ follow from the corresponding results for valued fields and hence we may assume that $r \ges 1$, so the derivation induced on $\bm k$ is nontrivial when $\bm k$ is $r$-linearly surjective, as was used in the preceding proofs.

\begin{thm}
If $K$ is asymptotic, $\Gamma$ is divisible, and $\bm k$ is $r$-linearly surjective, then any two $r$-$\d$-algebraically maximal $r$-$\d$-algebraic immediate extensions of $K$ are isomorphic over~$K$.
\end{thm}

The proof of \cite[Theorem~7.0.1]{adamtt} shows that if $K$ is $r$-$\d$-algebraically maximal and $\bm k$ is $r$-linearly surjective, then $K$ is $r$-$\d$-henselian.
Conversely:
\begin{thm}\label{rdhrdalgmax}
If $K$ is asymptotic and $r$-$\d$-henselian, and $\Gamma$ is divisible, then $K$ is $r$-$\d$-algebraically maximal.
\end{thm}

We say that an extension $L$ of $K$ is an \emph{$r$-differential-henselization} (\emph{$r$-$\d$-henselization} for short) of $K$ if it is an immediate asymptotic $r$-$\d$-henselian extension of $K$ that embeds over $K$ into any asymptotic $r$-$\d$-henselian extension of $K$.
In the next proof, instead of using \cite[Corollary~9.4.11]{adamtt} we let $K^{\dh}$ be any $r$-$\d$-algebraically maximal immediate $r$-$\d$-algebraic extension of $K$, since $K^{\dh}$ is $r$-$\d$-henselian and no proper differential subfield of $K^{\dh}$ containing $K$ is  $r$-$\d$-henselian by Theorem~\ref{rdhrdalgmax}.
%For the following, we need that \cite[Corollary~9.4.11]{adamtt} goes through with ``$r$-linearly surjective,'' ``$r$-$\d$-henselian,'' and ``$r$-$\d$-algebraic'' replacing ``linearly surjective,'' ``$\d$-henselian,'' and ``$\d$-algebraic'' respectively, if $\Gamma$ is divisible.
%To see this, if $K$ is as in Theorem~\ref{rdhensel} and $L$ is  any $r$-$\d$-algebraically maximal immediate $r$-$\d$-algebraic extension $K$, then $L$ is $r$-$\d$-henselian and has no proper $r$-$\d$-henselian differential subfield containing $K$ by Theorem~\ref{rdhrdalgmax}.

\begin{thm}\label{rdhensel}
If $K$ is asymptotic, $\Gamma$ is divisible, and $\bm k$ is $r$-linearly surjective, then $K$ has an $r$-$\d$-henselization, and any two $r$-$\d$-henselizations of $K$ are isomorphic over~$K$.
\end{thm}

\begin{cor}
If $K$ is asymptotic, $\Gamma$ is divisible, and $\bm k$ is $r$-linearly surjective, then any immediate $r$-$\d$-henselian extension of $K$ that is $r$-$\d$-algebraic over $K$ is an $r$-$\d$-henselization of $K$.
\end{cor}

\section{Newton diagrams}\label{sec:ndiag}

We develop a differential newton diagram method for valued differential fields with small derivation.
This approach is closely modelled on the differential newton diagram method for a certain class of asymptotic fields developed in \cite[\S13.5]{adamtt}.
In \S\ref{sec:ndiag:ddegcut}, we connect this to dominant degree in a cut, adapting two lemmas from \cite[\S13.6]{adamtt}.
The assumption of divisible value group allows us to use the Equalizer Theorem, which underlies this method:
\begin{thm}[{\cite[6.0.1]{adamtt}}]\label{adh6.0.1}
Let $P, Q \in K\{Y\}^{\neq}$ be homogeneous of degrees $m$, $n$, respectively, with $m>n$.
If $(m-n)\Gamma=\Gamma$, then there exists a unique $\alpha \in \Gamma$ such that $v_P(\alpha)=v_Q(\alpha)$.
\end{thm}

\begin{ass}
In this section, $K$ has a monomial group $\f{M}$.
\end{ass}
Let $P$ range over $K\{Y\}^{\neq}$, $f$ and $g$ over $K$, and $\f m$ and $\f n$ over $\f M$.
For $f \in K^{\x}$, let $\f d_f$ be the unique monomial with $\f d_f \asymp f$ and $u_f \coloneqq f/\f d_f \asymp 1$.
The next lemma plays a subtle but important role in this section and the next.
It is based on \cite[Lemma~13.5.4]{adamtt}, but with a somewhat different proof due to the differences between dominant parts and newton polynomials (the appropriate analogue of dominant parts in that setting).

\begin{lem}\label{ddegtrans}
Suppose that $\Gamma$ has no least positive element and $f \prece \f m$.
If $f \prec \f m$, let $u \coloneqq 0$; if $f \asymp \f m$, let $u \coloneqq u_f$.
Then
\[
\ddeg_{\prec \f m} P_{+f}\ =\ \mul(D_{P_{\xf m}})_{+\bar{u}}.
\]
In particular, $\ddeg_{\prec \f m} P = \dmul P_{\xf m}$.
\end{lem}
\begin{proof}
For $\f{n \prec m}$, let $\f{e=nm}^{-1} \in \f M$.
Then
\[
P_{+f, \xf n}\ =\ P_{\xf m, +\f{m}^{-1}f, \xf e},
\]
so by replacing $P$ with $P_{\xf m}$ and $f$ with $\f{m}^{-1}f$, we may assume that $\f m=1$.
Set $Q \coloneqq P_{+f}$, so by Lemma~\ref{adh6.6.5}(i), $D_Q = (D_P)_{+\bar{f}} = (D_P)_{+\bar{u}}$.
Thus $\mul(D_P)_{+\bar u}=\dmul Q$, so it remains to show
\[
\ddeg_{\prec 1} Q\ =\ \dmul Q.
\]

First, $\ddeg_{\prec 1} Q \les \dmul Q$ by Corollary~\ref{adh6.6.7}.
For the other direction, let $d \coloneqq \dmul Q$.
We have $v(Q_d) < v(Q_i)$ for all $i<d$, so take $g \prec 1$ with $vg$ small enough that
\[
v(Q_{d})+(d+1)\, vg\ <\ v(Q_i)
\qquad \text{for all}\ i<d.
\]
It follows that
\[
v(Q_{d})+d\, vg + o(vg)\ <\ v(Q_i)+i\, vg + o(vg)
\qquad \text{for all}\ i<d,
\]
so $v(Q_{d, \x g}) < v(Q_{i, \x g})$ for all $i<d$ by Lemma~\ref{adh6.1.3}.
Hence $\dmul Q_{\x g} \ges d$.
But
\[
\ddeg_{\prec 1} Q\ =\ \max \{\dmul Q_{\x g} : g \prec 1\}
\]
by Lemma~\ref{adh6.6.9}, so $\ddeg_{\prec 1} Q \ges d$, as desired.
\end{proof}

%In the next lemma, we use the notion of \emph{dominant weight} of a $\d$-polynomial $P$, denoted by $\dwt P$.
%This notion plays no role in the rest of the paper, so we refer the interested reader to \cite[\S4.5]{adamtt} for a definition.
%\begin{lem}\label{startmon}
%Let $f \in K^{\x}$, $\f m \coloneqq \f d_f$, and $u \coloneqq u_f$.
%Suppose $P(f)=0$.
%Then $D_{P_{\xf m}}(\bar{u})=0$, and thus $\dwt (P_{\xf m}) \ges 1$ or $D_{P_{\xf m}}$ is not homogeneous.
%\end{lem}
%\begin{proof}
%By replacing $P$ with $P_{\xf m}$ and $f$ with $\f m^{-1}f = u$, we may assume that $f=u \asymp 1$.
%By scaling $P$, we may assume that $P \asymp 1$.
%Then $D_P(\bar f)=\overline{P(f)}=0$.
%Now, if $\dwt(P)=0$, then $D_P \in \bm k[Y]^{\neq}$, but since it has a nonzero root in $\bm k$, it cannot be homogeneous.
%\end{proof}
%
%This lemma motivates several definitions.
%We call $y \in K^{\x}$ an \emph{approximate zero of $P$} if, for $\f m \coloneqq \f d_y$ and $u \coloneqq u_y$, $D_{P_{\xf m}}(\bar{u})=0$.
%If $y$ is an approximate zero of $P$, we define its \emph{multiplicity} to be $\mul(D_{P_{\xf m}})_{+\bar u}$.
%We say that $\f m$ is a \emph{starting monomial for $P$} if $\dwt(P_{\xf m}) \ges 1$ or $D_{P_{\xf m}}$ is not homogeneous.
%In particular, if $\f m$ is a starting monomial for $P$, then $\ddeg P_{\xf m} \ges 1$.
%The previous lemma shows that zeros of $P$ are approximate zeros of $P$, and also that if $y$ is an approximate zero of $P$, then $\f d_y$ is a starting monomial for $P$.

We call $y \in K^{\x}$ an \emph{approximate zero of $P$} if, for $\f m \coloneqq \f d_y$ and $u \coloneqq u_y$, $D_{P_{\xf m}}(\bar{u})=0$.
If $y$ is an approximate zero of $P$, we define its \emph{multiplicity} to be $\mul(D_{P_{\xf m}})_{+\bar u}$.
We call $\f m$ an \emph{algebraic starting monomial for $P$} if $D_{P_{\xf m}}$ is not homogeneous.
In particular, if $\f m$ is an algebraic starting monomial for $P$, then $\ddeg P_{\xf m} \ges 1$.
Note that $\f m$ is an algebraic starting monomial for $P$ if and only if $\f m/\f n$ is an algebraic starting monomial for $P_{\xf n}$.
By Corollary~\ref{adh6.6.7}, $P$ has at most $\deg P-\mul P$ algebraic starting monomials.

\begin{ass}
In the rest of this section, $\Gamma$ is divisible.
\end{ass}
The existence of algebraic starting monomials is an easy corollary of the Equalizer Theorem, and is crucial to what follows.
It corresponds to \cite[Corollary~13.5.6]{adamtt}.
\begin{lem}\label{equalizercor}
Let $P, Q \in K\{Y\}^{\neq}$ be homogeneous of different degrees.
Then there exists a unique $\f m$ such that $D_{(P+Q)_{\xf m}}$ is not homogeneous.
\end{lem}
\begin{proof}
By Theorem~\ref{adh6.0.1}, there is a unique $\f m$ such that $P_{\xf m} \asymp Q_{\xf m}$.
Then
\[
D_{(P+Q)_{\xf m}}\ =\ D_{P_{\xf m}+Q_{\xf m}}\ =\ D_{P_{\xf m}}+D_{Q_{\xf m}}
\]
by Lemma~\ref{adh6.6.2}(ii), so $D_{(P+Q)_{\xf m}}$ is not homogeneous.
For $\f n \neq \f m$, we have $D_{(P+Q)_{\xf n}}=D_{P_{\xf n}}$ or $D_{(P+Q)_{\xf n}}=D_{Q_{\xf n}}$ by Lemma~\ref{adh6.6.2}(i), since $P_{\xf n} \succ Q_{\xf n}$ or $P_{\xf n} \prec Q_{\xf n}$.
\end{proof}

For $P$ and $Q$ as in Lemma~\ref{equalizercor}, we let $\f e(P, Q)$ denote the unique monomial which that lemma yields and call it the \emph{equalizer for $P$, $Q$}.
We are interested in the case that these two $\d$-polynomials are homogeneous parts of the same $\d$-polynomial.
Let $J \coloneqq \{ j \in \N : P_j \neq 0 \}$ and note that $\ddeg P_{\xf m} \in J$ for all $\f m$.
For distinct $i, j \in J$, let $\f e(P, i,j) \coloneqq \f e(P_i, P_j)$,
and so any algebraic starting monomial for $P$ is of the form $\f e(P,i,j)$ for some distinct $i, j \in J$.
%Note that $\f M$ is (totally) ordered by~$\prece$.

In the next two results, let $\ca E \subseteq K^{\x}$ be $\prece$-closed.
Recall this means that $\ca E \neq \0$ and $f \in \ca E$ whenever $0 \neq f \prece g$ with $g \in \ca E$.
The first result corresponds to \cite[Proposition~13.5.7]{adamtt}.

\begin{prop}\label{ndiag}
There exist $i_0, \dots, i_n \in J$ and equalizers
\[
\f e(P,i_0,i_1)\ \prec\ \f e(P,i_1,i_2)\ \prec \dots \prec\ \f e(P,i_{n-1},i_n)
\]
with $\mul P = i_0 < \dots <i_n = \ddeg_{\ca E} P$ such that
\begin{enumerate}
\item the algebraic starting monomials for $P$ in $\ca E$ are the
$\f e(P,i_m,i_{m+1})$ for $m<n$;
\item for $m<n$ and $\f{m=e}(P,i_m,i_{m+1})$, we have $\dmul P_{\xf m}=i_m$ and $\ddeg P_{\xf m}=i_{m+1}$.
\end{enumerate}
\end{prop}
\begin{proof}
Let $i$, $j$ range over $J$ and $d \coloneqq \ddeg_{\ca E} P$.
Then $\mul P \les d \les \deg P$, and we proceed by induction on $d-\mul P$.
If $d = \mul P$, then for all $\f m \in \ca E$, $D_{P_{\xf m}}$ is homogeneous of degree $d$, so there is no algebraic starting monomial for $P$ in $\ca E$.

Now assume that $d > \mul P$ and take $i<d$ such that $\f e \coloneqq \f e(P,i,d) \succe \f e(P,j,d)$ for all $j<d$.
First, we show that $\f e \in \ca E$.
We have $P_{d, \xf e} \asymp P_{i, \xf e}$ by the previous lemma, so $v_{P_d}(v\f e)=v_{P_i}(v\f e)$.
By Lemma~\ref{adh6.1.3}, the function $v_{P_d}-v_{P_i}$ is strictly increasing, so for any $g \prec \f e$, we have $v_{P_d}(vg) > v_{P_i}(vg)$, that is, $P_{d, \x g} \prec P_{i, \x g}$.
Hence $\ddeg P_{\x g} < d$.
To obtain $\f e \in \ca E$, take $g \in \ca E$ with $\ddeg P_{\x g} = d$, so $\f e \prece g$.

Next, we show that $\ddeg P_{\xf e}=d$.
If $\ddeg P_{\xf e} = j < d$, then $P_{d, \xf e} \prec P_{j, \xf e}$.
By Lemma~\ref{adh6.1.3} again, the function $v_{P_d}-v_{P_j}$ is strictly increasing, so it follows that $\f e \prec \f e(P,j,d)$, contradicting the maximality of $\f e$.

From this and $P_{i, \xf e} \asymp P_{d, \xf e}$, we get $(D_{P_{\xf e}})_i = D_{P_{i, \xf e}} \neq 0$ and $(D_{P_{\xf e}})_d = D_{P_{d, \xf e}} \neq 0$, and hence $\f e$ is an algebraic starting monomial for $P$.
In fact, $\f e$ is the largest algebraic starting monomial for $P$ in $\ca E$.
Suppose to the contrary that $\f n \in \ca E$ is an algebraic starting monomial for $P$ with $\f{n \succ e}$.
Then $d=\ddeg P_{\xf e} \les \ddeg P_{\xf n}$ by Corollary~\ref{adh6.6.7}, so $\ddeg P_{\xf n}=d$.
It follows that $\f{n=e}(P,j,d)$ for some $j<d$, contradicting the maximality of $\f e$.

If $i > \dmul P_{\xf e}$, then for $j \coloneqq \dmul P_{\xf e}$, the uniqueness in Lemma~\ref{equalizercor} yields $\f e(P,j,d)=\f e$.
By replacing $i$ with $j$, we assume that $i = \dmul P_{\xf e}$.
Then by Lemma~\ref{ddegtrans}, we also have $\ddeg_{\prec \f e} P =i$.
To complete the proof, we apply the inductive assumption with $\{g \in K^{\x} : g \prec \f e\}$ replacing $\ca E$.
\end{proof}

The tuple $(i_0,\dots,i_n)$ from Proposition~\ref{ndiag} is uniquely determined by $K$, $P$, and $\ca E$.
Note that if $\mul P=\ddeg_{\ca E} P$, then $n=0$ and the tuple is $(\mul P)$.
For $1 \les m \les n$, set $\f e_m \coloneqq \f e(P,i_{m-1},i_m)$.
We now show how $\dmul P_{\x g}$ and $\ddeg P_{\x g}$ behave as $g$ ranges over $\ca E$.
This follows from Proposition~\ref{ndiag} in exactly the same way as \cite[Corollary~13.5.8]{adamtt} follows from \cite[Proposition~13.5.7]{adamtt}, using Corollary~\ref{adh6.6.7} instead of \cite[Corollary~11.2.5]{adamtt}.

\begin{cor}\label{ndiagcor}
Suppose that $\mul P \neq \ddeg_{\ca E} P$, so $n \ges 1$. Let $g$ range over $\ca E$.
Then $\dmul P_{\x g}$ and $\ddeg P_{\x g}$ are in $\{i_0,\dots,i_n\}$ and we have:
\begin{align*}
\dmul P_{\x g}\ =\ i_0 \quad &\iff \quad g\ \prece\ \f e_1;\\
\ddeg P_{\x g}\ =\ i_0 \quad &\iff \quad g\ \prec\ \f e_1;\\
\dmul P_{\x g}\ =\ i_m \quad &\iff \quad \f e_m\ \prec\ g\ \prece\ \f e_{m+1}, \qquad (1 \les m < n);\\
\ddeg P_{\x g}\ =\ i_m \quad &\iff \quad \f e_m\ \prece\ g\ \prec\ \f e_{m+1}, \qquad (1 \les m<n);\\
\dmul P_{\x g}\ =\ i_n \quad &\iff \quad \f e_n\ \prec\ g;\\
\ddeg P_{\x g}\ =\ i_n \quad &\iff \quad \f e_n\ \prece\ g.
\end{align*}
\end{cor}
%\begin{proof}
%We first prove the third equivalence, so let $1 \les m < n$.
%Then for $\f e_m \prec g \prec \f e_{m+1}$, Proposition~\ref{ndiag} and Corollary~\ref{adh6.6.7} give
%\[
%i_m\ =\ \ddeg P_{\xf e_m}\ \les\ \dmul P_{\x g}\ \les\ \ddeg P_{\x g}\ \les\ \dmul P_{\xf e_{m+1}}\ =\ i_m,
%\]
%which yields the right-to-left direction since if $g \asymp \f e_{m+1}$, then $\dmul P_{\x g} = \dmul P_{\xf e_{m+1}} = i_m$.
%For the converse, note that similarly, if $g \prece \f e_m$, then $\dmul P_{\x g} \les \dmul P_{\xf e_m} = i_{m-1}$, and if $g \succ \f e_{m+1}$, then $\dmul P_{\x g} \ges \ddeg P_{\xf e_{m+1}} = i_{m+1}$.
%The fourth equivalence is proved in the same way.
%
%For the first equivalence, if $g \prec \f e_1$, then
%\[
%i_0\ =\ \mul P\ \les\ \dmul P_{\x g}\ \les\ \ddeg P_{\x g}\ \les\ \dmul P_{\xf e_1}\ =\ i_0,
%\]
%and if $g \asymp \f e_1$, then $\dmul P_{\x g} = \dmul P_{\xf e_1} = i_0$.
%The converse follows as in the third equivalence.
%The remaining equivalences are proved similarly.
%\end{proof}

\subsection{Application to dominant degree in a cut}\label{sec:ndiag:ddegcut}

The first lemma is easily adapted from \cite[Lemma~13.6.17]{adamtt} and its corollary corresponds to \cite[Corollary~13.6.18]{adamtt}.
\begin{lem}\label{ndiagddeg}
Suppose that $(a_\rho)$ is a pc-sequence in $K$ with $a_\rho \pconv 0$.
Let
\[
\ca E_{\bm a}\ \coloneqq\ \{g \in K^{\x}: g\prec a_\rho,\ \text{eventually}\}.
\]
\begin{enumerate}
	\item If $\ca E_{\bm a} \neq \0$, then $\ddeg_{\bm a} P = \ddeg_{\ca E_{\bm a}} P$.
	\item If $\ca E_{\bm a}=\0$, then $\ddeg_{\bm a}P=\mul P$.
\end{enumerate}
\end{lem}
\begin{proof}
Set $\gamma_\rho \coloneqq v(a_{\rho+1}-a_\rho)$.
By removing some initial terms, we may assume that $\gamma_\rho$ is strictly increasing and $v(a_\rho) = \gamma_\rho \in \Gamma$ for all $\rho$.
Then by Lemma~\ref{adh6.6.10},
\[
\ddeg_{\ges \gamma_\rho} P_{+a_\rho}\ =\ \ddeg_{\ges\gamma_\rho} P\ =\ \ddeg P_{\x a_\rho},
\]
so $\ddeg_{\bm a} P$ is the eventual value of $\ddeg P_{\x a_\rho}$.
If $P$ is homogeneous, then $\ddeg P_{\x g} = \deg P = \mul P$ for all $g \in K^{\x}$, so the statements about $\ddeg_{\bm a} P$ are immediate.

Suppose now that $P$ is not homogeneous, so $\mul P<\deg P$.
If $\ca E_{\bm a} \neq \0$, we use Corollary~\ref{ndiagcor} with $K^{\x}$ in the role of $\ca E$, so we have the tuple $(i_0, \dots, i_n)$ with $i_n=\deg P$.
By removing further initial terms, we may assume that $\ddeg P_{\x a_\rho}$ is constant.
If $\ddeg P_{\x a_\rho}=i_0$, then $a_{\rho} \prec \f e_1$.
Thus for any $g \in \ca E_{\bm a}$, we have $g\prec \f e_1$, and hence $\ddeg_{\ca E_{\bm a}} P=i_0$.
If $\ddeg P_{\x a_\rho}=i_m$ for any $1 \les m \les n$, then $\f e_m \prece a_\rho$.
As $\gamma_\rho$ is strictly increasing, $\f e_m \prec a_\rho$ for all $\rho$, so $\f e_m \in \ca E_{\bm a}$.
Hence $\ddeg_{\ca E_{\bm a}} P \ges \ddeg P_{\xf e_m} = i_m$.
But by Corollary~\ref{adh6.6.7}, $\ddeg P_{\x a_\rho} \ges \ddeg_{\ca E_{\bm a}} P$, so $\ddeg_{\ca E_{\bm a}} P = i_m$.

If $\ca E_{\bm a}=\0$, %then $\gamma_\rho$ is cofinal in $\Gamma$, and thus $(a_\rho)$ is a cauchy sequence (see \cite[\S2.2]{adamtt}).
let $i_0 \coloneqq \mul P$.
Then for all $i>i_0$, by Lemma~\ref{adh6.1.3},
\[
v_{P_i}(\gamma_\rho) - v_{P_{i_0}}(\gamma_\rho)\ =\ v(P_i)-v(P_{i_0}) + (i-i_0)\gamma_\rho + o(\gamma_\rho).
\]
As $\gamma_\rho$ is cofinal in $\Gamma$, we thus have $v_P(\gamma_\rho)=v_{P_{i_0}}(\gamma_\rho)<v_{P_i}(\gamma_\rho)$, eventually, for all $i>i_0$, and so $\ddeg P_{\x a_\rho}=i_0$, eventually.
\end{proof}

With $(a_\rho)$ and $\ca E_{\bm a}$ as in the above lemma, if $\ca E_{\bm a} = \0$, then $(a_\rho)$ is in fact a cauchy sequence in $K$ (see \cite[\S2.2]{adamtt}), since $\gamma_\rho$ is cofinal in $\Gamma$; this is not used later.

\begin{cor}\label{ndiagddegcor}
Suppose $(b_\rho)$ is a pc-sequence in $K$ with pseudolimit $b \in K$.
Let $\bm b \coloneqq c_K(b_\rho)$ and
\[
\ca E_{\bm b}\ \coloneqq\ \{g \in K^{\x}: g \prec b_\rho-b,\ \text{eventually}\}.
\]
\begin{enumerate}
	\item If $\ca E_{\bm b}\neq \0$, then $\ddeg_{\bm b} P = \ddeg_{\ca E_{\bm b}} P_{+b}$.
	\item If $\ca E_{\bm b}=\0$, then $\ddeg_{\bm b} P = \mul P_{+b}$.
\end{enumerate}
\end{cor}
\begin{proof}
This follows by applying Lemma~\ref{ndiagddeg} to $a_\rho \coloneqq b_\rho-b$ and $P_{+b}$, using Lemma~\ref{ddegbasic}(ii).
\end{proof}
%\begin{proof}
%Set $a_\rho \coloneqq b_\rho-b$.
%By Lemma~\ref{ddegbasic}(ii), we have
%\[
%\ddeg_{\bm b} P\ =\ \ddeg_{\bm a+b} P\ =\ \ddeg_{\bm a} P_{+b}.
%\]
%It remains to apply the previous lemma with $P_{+b}$ in place of $P$.
%\end{proof}

\section{Asymptotic differential equations}\label{sec:ade}

\begin{ass}
In this section, $K$ has a monomial group $\f{M}$ and $\Gamma$ has no least positive element.
\end{ass}

Let $\f m$ range over $\f M$ and $P \in K\{Y\}^{\neq}$ have order at most $r$.
An \emph{asymptotic differential equation} over $K$ is something of the form
\begin{equation}\tag{E}\label{eqn}
P(Y)\ =\ 0, \hs Y \in \ca E,
\end{equation}
where $\ca E \subseteq K^\x$ is $\prece$-closed.
That is, it consists of an algebraic differential equation with an asymptotic condition on solutions.
If $\ca E = \{ g \in K^{\x} : g \prec f \}$ for some $f \in K^{\x}$, then we write $Y \prec f$ for the asymptotic condition instead of $Y \in \ca E$, and similarly with ``$\prece$.''

For the rest of this section, we fix such an asymptotic differential equation \eqref{eqn}.
Then the \emph{dominant degree of \eqref{eqn}} is defined to be $\ddeg_{\ca E} P$.
A \emph{solution of \eqref{eqn}} is a $y \in \ca E$ such that $P(y)=0$.
An \emph{approximate solution of \eqref{eqn}} is an approximate zero of $P$ that lies in $\ca E$, and the \emph{multiplicity} of an approximate solution of \eqref{eqn} is its multiplicity as an approximate zero of $P$.
The following is used frequently and follows from Lemma~\ref{ddegtrans}.

\begin{cor}\label{approxsolnequiv}
Let $y \in \ca E$.
Then
\begin{enumerate}
\item $y$ is an approximate solution of \eqref{eqn} $\iff$
$\ddeg_{\prec y} P_{+y} \ges 1$;
\item if $y$ is an approximate solution of \eqref{eqn}, then its multiplicity is
$\ddeg_{\prec y} P_{+y}$.
\end{enumerate}
\end{cor}

%A \emph{starting monomial for \eqref{eqn}} is a starting monomial for $P$ that lies in $\ca E$.
An \emph{algebraic starting monomial for \eqref{eqn}} is an algebraic starting monomial for $P$ that lies in $\ca E$.
%If $y$ is an approximate solution of \eqref{eqn}, then $\f d_y$ is a starting monomial for \eqref{eqn} by Lemma~\ref{startmon}.
%In particular, if $\ddeg_{\ca E} P=0$, then \eqref{eqn} has no approximate solutions.
So if $\ddeg_{\ca E} P=0$, then \eqref{eqn} has no algebraic starting monomials.
By Proposition~\ref{ndiag}, if $\Gamma$ is divisible and $\mul P<\ddeg_{\ca E} P$, then there is an algebraic starting monomial for \eqref{eqn} and $\ddeg_{\ca E} P = \ddeg P_{\xf e}$, where $\f e$ is the largest algebraic starting monomial for \eqref{eqn}.

It will be important to alter $P$ and $\ca E$ in certain ways.
Namely, let $\ca{E' \subseteq E}$ be $\prece$-closed and let $f \in \ca E \cup\{0\}$.
We call the asymptotic differential equation
\begin{equation}\tag{E${}'$}\label{eqn'}
P_{+f}(Y)\ =\ 0, \hs Y \in \ca E'
\end{equation}
a \emph{refinement of \eqref{eqn}}.
Below, \eqref{eqn'} refers to a refinement of this form.
By Lemma~\ref{adh6.6.10},
\[
\ddeg_{\ca E} P\ =\ \ddeg_{\ca E} P_{+f}\ \ges\ \ddeg_{\ca E'} P_{+f},
\]
so the dominant degree of \eqref{eqn'} is at most the dominant degree of \eqref{eqn}.
Note also that if $y$ is a solution of \eqref{eqn'} and $f+y \neq 0$, then $f+y$ is a solution of \eqref{eqn}.
The same is true with ``approximate solution'' replacing ``solution,'' provided that $y \not\sim -f$.
%\begin{lem}
%Let $y \not\sim -f$ be an approximate solution of \eqref{eqn'} of multiplicity $\mu$.
%Then $f+y$ is an approximate solution of \eqref{eqn} of multiplicity at least $\mu$.
%\end{lem}
%\begin{proof}
%From $y \not\sim -f$, we get $y \prece f+y$, so
%\[
%1\ \les\ \mu\ =\ \ddeg_{\prec y}(P_{+f})_{+y}\ =\ \ddeg_{\prec y}P_{+f+y}\ \les\ \ddeg_{\prec f+y}P_{+f+y}.
%\]
%That is, $f+y$ is an approximate solution of \eqref{eqn} of multiplicity at least $\mu$.
%\end{proof}
The following routine adaptation of \cite[Lemma~13.8.2]{adamtt} gives a sufficient condition for being an approximate solution.

\begin{lem}\label{approxsolncheck}
Let $f \neq 0$ with $f \succ g$ for all $g \in \ca E'$,
and suppose that
\[
\ddeg_{\ca E'}P_{+f}\ =\ \ddeg_{\ca E}P\ \ges\ 1.
\]
Then $f$ is an approximate solution of \eqref{eqn}.
\end{lem}
\begin{proof}
We have, using Lemma~\ref{adh6.6.10} for the equality,
\[
\ddeg_{\ca E'} P_{+f}\ \les\ \ddeg_{\prec f} P_{+f}\ \les\ \ddeg_{\prece f} P_{+f}\ =\ \ddeg_{\prece f} P\ \les\ \ddeg_{\ca E} P.
\]
Hence $\ddeg_{\prec f} P_{+f} = \ddeg_{\ca E}P \ges 1$, so $f$ is an approximate solution of \eqref{eqn}.
\end{proof}

Note that by the previous proof, $\ddeg_{\prec f} P_{+f} \les \ddeg_{\ca E} P$ for all $f \in \ca E$.
The next lemma, corresponding to \cite[Lemma~13.8.3]{adamtt}, relates the strictness of this inequality to approximate solutions.

\begin{lem}\label{unraveldef}
Suppose that $d \coloneqq \ddeg_{\ca E} P \ges 1$. Then the following are equivalent:
\begin{enumerate}
\item $\ddeg_{\prec f}P_{+f}<d$ for all $f \in \ca E$;
\item $\ddeg_{\prec f}P_{+f}<d$ for all $f \in \ca E$ with $\ddeg P_{\x f}=d$;
\item there is no approximate solution of \eqref{eqn} of multiplicity $d$.
\end{enumerate}
\end{lem}
\begin{proof}
The equivalence of (i) and (iii) is given by Corollary~\ref{approxsolnequiv}.
Now, let $f \in \ca E$ and suppose that $\ddeg P_{\x f}<d$.
Then, using Lemma~\ref{adh6.6.10} for the first equality,
\[
\ddeg_{\prec f} P_{+f}\ \les\ \ddeg_{\prece f} P_{+f}\ =\ \ddeg_{\prece f} P\ =\ \ddeg P_{\x f}\ <\ d.
\]
This gives (ii)$\implies$(i), and the converse is trivial.
\end{proof}

We say that \eqref{eqn} is \emph{unravelled} if $d \coloneqq \ddeg_{\ca E} P \ges 1$ and the conditions in Lemma~\ref{unraveldef} hold.
In particular, if $d \ges 1$ and \eqref{eqn} does not have an approximate solution, then \eqref{eqn} is unravelled.
And if \eqref{eqn} is unravelled and has an approximate solution, then $d \ges 2$ by Lemma~\ref{unraveldef}(iii).
We now introduce unravellers and partial unravellers, which correspond to special refinements of \eqref{eqn}.
In the proof of Proposition~\ref{unravellersexist}, we construct a sequence of partial unravellers ending in an unravelled asymptotic differential equation.
Suppose that $d \ges 1$, and let $f \in \ca E\cup\{0\}$ and $\ca E' \subseteq \ca E$ be $\prece$-closed.
We say that $(f, \ca E')$ is a \emph{partial unraveller for \eqref{eqn}} if $\ddeg_{\ca E'}P_{+f}=d$.
By Lemma~\ref{adh6.6.10}, $(f, \ca E)$ is a partial unraveller for \eqref{eqn}.
Note that if $(f, \ca E')$ is a partial unraveller for \eqref{eqn} and $(f_1, \ca E_1)$ is a partial unraveller for \eqref{eqn'}, then $(f+f_1, \ca E_1)$ is a partial unraveller for \eqref{eqn}.
An \emph{unraveller for \eqref{eqn}} is a partial unraveller $(f, \ca E')$ for \eqref{eqn} with unravelled \eqref{eqn'}.
The following is routine, corresponding to \cite[Lemma~13.8.6]{adamtt}.
\begin{lem}\label{multconjade}
Suppose that $\ddeg_{\ca E} P \ges 1$.
Let $a \in K^{\x}$ and set $a\ca E \coloneqq \{ay \in K^{\x} : y \in \ca E\}$.
Consider the asymptotic differential equation
\begin{equation}\tag{$a$\ref{eqn}}\label{aeqn}
P_{\x a^{-1}}(Y)\ =\ 0, \hs Y \in a\ca E.
\end{equation}
\begin{enumerate}
	\item The dominant degree of \eqref{aeqn} equals the dominant degree of \eqref{eqn}.
	\item If $(f, \ca E')$ is a partial unraveller for \eqref{eqn}, then $(af, a\ca E')$ is a partial unraveller for \eqref{aeqn}.
	\item If $(f, \ca E')$ is an unraveller for \eqref{eqn}, then $(af, a\ca E')$ is an unraveller for \eqref{aeqn}.
	\item If $a \in \f M$, then the algebraic starting monomials for \eqref{aeqn} are exactly the $a\f e$, where $\f e$ ranges over the algebraic starting monomials for \eqref{eqn}.
\end{enumerate}
\end{lem}

%The next result follows immediately from Lemma~\ref{unravel+f}.
%\begin{cor}
%Suppose $\ddeg_{\ca E} P \ges 1$ and $(f, \ca E')$ is an unraveller for \eqref{eqn}.
%Let $g \in K$ such that $f-g \prece \f e$ for some algebraic starting monomial $\f e$ for \eqref{eqn'}.
%Then $(g, \ca E')$ is an unraveller for \eqref{eqn}.
%\end{cor}
%%\begin{proof}
%%This is routine from Lemma~\ref{unravel+f}.
%%\end{proof}

The next proposition is about the existence of unravellers, and is a key ingredient in the proof of Proposition~\ref{mainlemmadiv}.
It corresponds to \cite[Proposition~13.8.8]{adamtt}, and the main difference in the proof is to invoke $r$-$\d$-algebraic maximality and the nontriviality of the derivation on $\bm k$ instead of asymptotic $\d$-algebraic maximality to obtain pseudolimits of appropriate pc-sequences.
For this, recall the notion of $r$-$\d$-algebraic maximality from \S\ref{sec:main:add}, and in particular that if the derivation induced on $\bm k$ is nontrivial, then $K$ is $r$-$\d$-algebraically maximal if and only if every pc-sequence in $K$ with minimal $\d$-polynomial over $K$ of order at most $r$ has a pseudolimit in $K$.

\begin{prop}\label{unravellersexist}
Suppose that $K$ is $r$-$\d$-algebraically maximal, $\Gamma$ is divisible, and the derivation induced on $\bm k$ is nontrivial.
Suppose that
$d \coloneqq \ddeg_{\ca E}P \ges 1$ and that there is no $f \in \ca E\cup\{0\}$ with $\mul P_{+f}=d$.
Then there exists an unraveller for \eqref{eqn}.
\end{prop}
\begin{proof}
We construct a sequence $\big((f_\lambda, \ca E_\lambda)\big)_{\lambda<\rho}$ of partial unravellers for \eqref{eqn} indexed by an ordinal $\rho>0$ such that:
\begin{enumerate}
\item $\ca{E_\lambda \supseteq E_\mu}$ for all $\lambda<\mu<\rho$;
\item $f_\mu-f_\lambda \succ f_\nu-f_\mu$ for all $\lambda<\mu<\nu<\rho$;
\item $f_{\lambda+1}-f_\lambda \in \ca{E}_\lambda \setminus \ca{E}_{\lambda +1}$ for all $\lambda$ with $\lambda+1<\rho$.
\end{enumerate}
For $\rho=1$, we set $(f_0, \ca E_0) \coloneqq (0, \ca E)$ and these conditions are vacuous.
Below, we frequently use that by (ii) we have $f_\mu-f_\lambda \asymp f_{\lambda+1}-f_\lambda$ for all $\lambda<\mu<\rho$.

First, suppose that $\rho$ is a successor ordinal, so $\rho=\sigma+1$, and consider the refinement 
\begin{equation}\tag{E${}_\sigma$}\label{eqnsig}
P_{+f_\sigma}(Y)\ =\ 0, \hs Y \in \ca E_\sigma
\end{equation}
of \eqref{eqn}.
If \eqref{eqnsig} is unravelled, then $(f_\sigma, \ca E_\sigma)$ is an unraveller for \eqref{eqn} and we are done, so suppose that \eqref{eqnsig} is not unravelled.
Take $f \in \ca E_\sigma$ such that $\ddeg_{\prec f}(P_{+f_\sigma})_{+f}=d$.
Then
\[
\ca E_\rho\ \coloneqq\ \{y \in K^{\x} : y \prec f\}\ \subset\ \ca E_\sigma
\]
is $\prece$-closed with
\[
\ddeg_{\ca E_\rho} (P_{+f_\sigma})_{+f}\ =\ d,
\]
so $(f, \ca E_\rho)$ is a partial unraveller for \eqref{eqnsig}.
Thus, setting $f_\rho \coloneqq f_\sigma+f$, we have that $(f_\rho, \ca E_\rho)$ is a partial unraveller for \eqref{eqn}.
Conditions (i) and (iii) on $\big((f_\lambda, \ca E_\lambda)\big)_{\lambda<\rho+1}$ with $\rho+1$ in place of $\rho$ are obviously satisfied.
For (ii), it is sufficient to check that $f_{\lambda+1}-f_\lambda \succ f_\rho - f_\sigma = f$ for $\lambda<\sigma$, which follows from $f_{\lambda+1}-f_{\lambda} \notin \ca E_\sigma$.

Now suppose that $\rho$ is a limit ordinal.
By (ii), $(f_\lambda)_{\lambda<\rho}$ is a pc-sequence in $K$, so we let $\bm f \coloneqq c_K(f_\lambda)$ and claim that $\ddeg_{\bm f} P = d$.
To see this, set $g_{\lambda} \coloneqq f_{\lambda+1}-f_\lambda$ for $\lambda$ with $\lambda+1<\rho$.
By (iii), we have, using Lemma~\ref{adh6.6.10}  in the third line,
\begin{align*}
d\ =\ \ddeg_{\ca E_{\lambda+1}} P_{+f_{\lambda+1}}\ &\les\ \ddeg_{\prece g_\lambda} P_{+f_{\lambda+1}}\\
&=\ \ddeg_{\prece g_\lambda} (P_{+f_\lambda})_{+(f_{\lambda+1}-f_\lambda)}\\
&=\ \ddeg_{\prece g_\lambda} P_{+f_\lambda}\\
&\les\ \ddeg_{\ca E_\lambda} P_{+f_\lambda}\ =\ d.
\end{align*}
Thus $\ddeg_{\prece g_\lambda} P_{+f_\lambda}=d$ for all $\lambda<\rho$, so $\ddeg_{\bm f} P = d$.
By Lemma~\ref{ddegcutfo} and Corollary~\ref{mindiffpolyfo}, $(f_\lambda)_{\lambda<\rho}$ has a minimal $\d$-polynomial over $K$ of order at most $r$, so since $K$ is $r$-$\d$-algebraically maximal, we may take $f_\rho \in K$ with $f_\lambda \pconv f_\rho$.
Now set
\[
\ca E_\rho\ \coloneqq\ \bigcap_{\lambda<\rho} \ca E_\lambda\ =\ \left\{y \in K^{\x}:y \prec g_\lambda\ \text{for all}\ \lambda<\rho\right\},
\]
where the equality follows from (iii).
If $\ca E_\rho =\0$, then by Corollary~\ref{ndiagddegcor},
\[
d\ =\ \ddeg_{\bm f} P\ =\ \mul P_{+f_\rho},
\]
contradicting the hypothesis.
So $\ca E_\rho\neq\0$, and thus Corollary~\ref{ndiagddegcor} yields
\[
d\ =\ \ddeg_{\bm f} P\ =\ \ddeg_{\ca E_\rho} P_{+f_\rho},
\]
so $(f_\rho, \ca E_\rho)$ is a partial unraveller for \eqref{eqn}.
For $\big((f_\lambda, \ca E_\lambda)\big)_{\lambda<\rho+1}$, conditions (i) and (iii) with $\rho+1$ in place of $\rho$ are obviously satisfied.
For (ii), it is enough to check that $f_{\lambda+1}-f_\lambda \succ f_\rho - f_\mu$ for $\lambda<\mu<\rho$, which follows from $f_\rho - f_\mu \asymp f_{\mu+1}-f_\mu$.

This inductive construction must end, and therefore there exists an unraveller for \eqref{eqn}.
\end{proof}

\subsection{Behaviour of unravellers under immediate extensions}\label{secade:unravelimmed}

In this subsection, we fix an immediate extension $L$ of $K$, and we use the monomial group of $K$ as a monomial group for $L$.
We consider how unravellers change under passing from $K$ to $L$ and connect this to pseudolimits of pc-sequences.
Lemma~\ref{plimunravellersexist} is a key step in the proof of Proposition~\ref{mainlemmadiv}.

Given $\ca E$, the set $\ca E_L \coloneqq \{y \in L^{\x}:vy \in v\ca E\}$ is also $\prece$-closed with $\ca E_L \cap K=\ca E$.
Consider the asymptotic differential equation
\begin{equation}\tag{E${}_L$}\label{eqnL}
P(Y)\ =\ 0, \hs Y \in \ca E_L
\end{equation}
over $L$, which has the same dominant degree as \eqref{eqn}, i.e., $\ddeg_{\ca E_L} P = \ddeg_{\ca E}P$.
Note that $y \in K$ is an approximate solution of \eqref{eqn} if and only if it is an approximate solution of \eqref{eqnL}.
If so, its multiplicities in both settings agree.
Thus if \eqref{eqnL} is unravelled, then \eqref{eqn} is unravelled.
For the other direction, if $y \in L$ is an approximate solution of \eqref{eqnL} of  multiplicity $\ddeg_{\ca E_L} P$, then any $z \in K$ with $z \sim y$ is an approximate solution of \eqref{eqn} of multiplicity $\ddeg_{\ca E} P = \ddeg_{\ca E_L} P$.
The next lemma follows from this, and corresponds to \cite[Lemma~13.8.9]{adamtt}.

\begin{lem}\label{unravelimmedlem}
Suppose that $\ddeg_{\ca E}P \ges 1$, and let $f \in \ca E\cup\{0\}$ and $\ca {E' \subseteq E}$ be $\prece$-closed.
Then:
\begin{enumerate}
	\item $(f, \ca E')$ is a partial unraveller for \eqref{eqn} if and only if $(f, \ca E'_L)$ is a partial unraveller for \eqref{eqnL};
	\item $(f, \ca E')$ is an unraveller for \eqref{eqn} if and only if $(f, \ca E'_L)$ is an unraveller for \eqref{eqnL}.
\end{enumerate}
\end{lem}

This next lemma does not use the assumptions of this section.
It corresponds to \cite[Lemma~13.8.10]{adamtt} and has exactly the same proof, except for using Lemma~\ref{adh6.8.1} instead of \cite[Lemma~11.3.8]{adamtt}.
\begin{lem}\label{mul}
Suppose that the derivation induced on $\bm k$ is nontrivial.
Let $(a_\rho)$ be a divergent pc-sequence in $K$ with minimal $\d$-polynomial $G$ over $K$, and let $a_\rho \pconv \ell \in L$.
Then $\mul(G_{+\ell}) \les 1$.
\end{lem}
%\begin{proof}
%Let $H \in K\{Y\}^{\neq}$ be of lower complexity than $G$.
%If $H(\ell)=0$, then by Lemma~\ref{adh6.8.1} there would be an equivalent pc-sequence $(b_\rho)$ in $K$ with $H(b_\rho) \pconv 0$, contradicting the minimality of $G$.
%
%In particular, $S_G(\ell)\neq 0$, where $S_G \coloneqq \partial G/\partial Y^{(n)}$ is the separant of $G$ and $n$ is the order of $G$.
%To see that $\mul (G_{+\ell}) \les 1$, decompose
%\[
%G_{+\ell}\ =\ \sum_{i=0}^m F_i \cdot \big(Y^{(n)}\big)^i \qquad \text{and} \qquad S_{G_{+\ell}}\ =\ \sum_{i=1}^m i F_i \cdot \big(Y^{(n)}\big)^{i-1},
%\]
%with $F_i \in K[Y, \dots, Y^{(n-1)}]$, $i=0,\dots,m$.
%From $S_{G_{+\ell}}=(S_G)_{+\ell}$ and $S_G(\ell) \neq 0$ we get $F_1(0) \neq 0$, so $\mul F_1 = 0$, and thus $\mul G_{+\ell} \les 1$.
%\end{proof}

The next lemma is routinely adapted from \cite[Lemma~13.8.11]{adamtt}.
\begin{lem}\label{plimunravellersexist}
Suppose that $\Gamma$ is divisible and the derivation induced on $\bm k$ is nontrivial.
Let $(a_\rho)$ be a divergent pc-sequence in $K$ with minimal $\d$-polynomial $P$ over $K$, and $a_\rho \pconv \ell \in L$.
Suppose that $L$ is $r$-$\d$-algebraically maximal and $\ddeg_{\bm a}P\ges 2$.
Let $a \in K$ and $\f v \in K^{\x}$ be such that $a-\ell \prec \f v$ and $\ddeg_{\prec \f v} P_{+a} = \ddeg_{\bm a}P$. \textnormal{(}Such $a$ and $\f v$ exist by Lemma~\ref{ddegcutfo}.\textnormal{)}
Consider the asymptotic differential equation
\begin{equation}\label{eqnav}
P_{+a}(Y)\ =\ 0, \hs Y\prec \f v.
\end{equation}
Then there exists an unraveller $(f, \ca E)$ for \eqref{eqnav} over $L$ such that:
\begin{enumerate}
	\item $f \neq 0$;
	\item $\ddeg_{\prec f} P_{+a+f} = \ddeg_{\bm a} P$;
	\item $a_\rho \pconv a+f+z$ for all $z \in \ca E\cup\{0\}$.
\end{enumerate}
\end{lem}
\begin{proof}
We first show how to arrange that $a=0$ and (ii) holds.
Take $g \in K^{\x}$ with $a-\ell \sim -g$, so $g\prec \f v$.
Then, using Lemma~\ref{adh6.6.10}, we have
\[
\ddeg_{\prec g} P_{+a+g}\ \les\ \ddeg_{\prec \f v} P_{+a+g}\ =\ \ddeg_{\prec \f v} P_{+a}\ =\ \ddeg_{\bm a} P.
\]
Conversely, as $(a+g)-\ell \prec g$, Lemma~\ref{ddegcutfo} gives $\ddeg_{\bm a}P \les \ddeg_{\prec g} P_{+a+g}$, so
\[\ddeg_{\bm a} P\ =\ \ddeg_{\prec \f v} P_{+a}\ =\ \ddeg_{\prec g} P_{+a+g}.\]
Also, $P_{+a+g}$ is a minimal $\d$-polynomial of $\big(a_\rho-(a+g)\big)$ over $K$ and, by Lemma \ref{ddegbasic}(ii),
\[
\ddeg_{\bm a-(a+g)}P_{+a+g}\ =\ \ddeg_{\bm a} P.
\]
We can now replace $P$, $(a_\rho)$, $\ell$,  and $\f v$ with $P_{+a+g}$, $\big(a_\rho-(a+g)\big)$, $\ell-(a+g)$, and $g$, respectively.
To see that this works, suppose that $\ca E \subseteq L^{\x}$ is $\prece$-closed in $L$ with $\ca E \prec g$, and $(h, \ca E)$ is an unraveller for the asymptotic differential equation
\[
P_{+a+g}(Y)\ =\ 0, \hs Y\prec g
\]
over $L$ with $a_\rho-(a+g) \pconv h+z$ for all $z \in \ca E\cup\{0\}$.
In particular, $h \prec g$, so $g+h \neq 0$, and it is clear from $\ddeg_{\prec \f v} P_{+a} = \ddeg_{\prec g} P_{+a+g}$ that $(g+h, \ca E)$ is an unraveller for \eqref{eqnav}.
Condition (iii) is also obviously satisfied.
For condition (ii), note that as $h \prec g$, using Lemma~\ref{adh6.6.10} in the middle equality,
\[
\ddeg_{\prec g+h} P_{+a+g+h}\ =\ \ddeg_{\prec g} P_{+a+g+h}\ =\ \ddeg_{\prec g} P_{+a+g}\ =\ \ddeg_{\bm a} P.
\]

Thus it remains to show that there is an unraveller $(f, \ca E)$ for \eqref{eqnav} in $L$ (with $a=0$) such that $a_\rho \pconv f+z$ for all $z \in \ca E\cup\{0\}$.
Consider the set
\[
\ca Z\ \coloneqq\ \{z \in L^{\x}:z\prec a_\rho-\ell,\ \text{eventually}\}.
\]
For any $z \in \ca Z\cup\{0\}$, we have $a_\rho\pconv z+\ell$, so by Lemma~\ref{mul}, \[
\mul (P_{+\ell+z})\ \les\ 1\ <\ 2\ \les\ \ddeg_{\bm a} P.
\]
By Corollary~\ref{ndiagddegcor}, $\ca Z \neq\0$, so $\ca Z$ is $\prece$-closed and $\ddeg_{\ca Z} P_{+\ell}=\ddeg_{\bm a} P$.
Then Proposition~\ref{unravellersexist} yields an unraveller $(s, \ca E)$ for the asymptotic differential equation
\[
P_{+\ell}(Y)\ =\ 0, \hs Y \in \ca Z
\]
over $L$.
Setting $f \coloneqq \ell+s$, we get that $(f, \ca E)$ is an unraveller for \eqref{eqnav} with $a_\rho \pconv f+z$ for all $z \in \ca E\cup\{0\}$.
\end{proof}

\subsection{Reducing degree}

In this subsection, we consider a refinement of \eqref{eqn} and then truncate it by removing monomials of degree higher than the dominant degree of \eqref{eqn}.
Given an unraveller for \eqref{eqn}, we show how to find an unraveller for this truncated refinement in Lemma~\ref{reducedegunravel}, an essential component in the proof of Proposition~\ref{maintechprop}.

\begin{ass}
In this subsection, $\Gamma$ is divisible.
\end{ass}
Suppose that $d \coloneqq \ddeg_{\ca E} P \ges 1$ and we have an unraveller $(f, \ca E')$ for \eqref{eqn}.
That is, the refinement
\begin{equation}\tag{\ref{eqn'}}\label{eqn'2}
P_{+f}(Y)\ =\ 0, \hs Y\in \ca E'
\end{equation}
of \eqref{eqn} is unravelled with dominant degree $d$.
Now suppose that $d>\mul (P_{+f})$, so \eqref{eqn'} has an algebraic starting monomial, and let $\f e$ be its largest algebraic starting monomial.
Suppose that $g \in K^{\x}$ satisfies $\f e\prec g \prec f$, and consider another refinement of \eqref{eqn}:
\begin{equation}\tag{E${}_g$}\label{eqng}
P_{+f-g}(Y)\ =\ 0, \hs Y \prece g.
\end{equation}
Set $\ca E'_g \coloneqq \{y \in \ca E': y \prec g\}$, so $\f e \in \ca E'_g$.
The next lemma is routinely adapted from \cite[Lemma~13.8.12]{adamtt}.

\begin{lem}\label{tildeunravel}
The asymptotic differential equation \eqref{eqng} has dominant degree $d$ and $(g, \ca E'_g)$ is an unraveller for \eqref{eqng}.
\end{lem}
\begin{proof}
First, since $\f e$ is the largest algebraic starting monomial for \eqref{eqn'2}, Proposition~\ref{ndiag} gives
\[
d\ =\ \ddeg P_{+f, \xf e}\ =\ \ddeg_{\prece \f e} P_{+f}.
\]
Note that $f-g \sim f \in \ca E$.
Now, by Lemma~\ref{adh6.6.10} we obtain
\[
d\ =\ \ddeg_{\prece \f e} P_{+f}\ \les\ \ddeg_{\prece g} P_{+f}\ =\ \ddeg_{\prece g} P_{+f-g}\ \les\ \ddeg_{\ca E} P_{+f-g}\ =\ \ddeg_{\ca E} P\ =\ d,
\]
which gives that \eqref{eqng} has dominant degree $d$.
Similarly,
\[
d\ =\ \ddeg_{\prece \f e} P_{+f}\ \les\ \ddeg_{\ca E'_g} P_{+f}\ \les\ \ddeg_{\ca E} P_{+f}\ =\ \ddeg_{\ca E} P\ =\ d,
\]
and thus the asymptotic differential equation
\[
P_{+f}(Y)\ =\ 0, \hs Y \in \ca E'_g,
\]
which is a refinement of both \eqref{eqng} and \eqref{eqn'2}, has dominant degree $d$.
Finally, since \eqref{eqn'2} is unravelled, the pair $(g, \ca E'_g)$ is an unraveller for \eqref{eqng}.
\end{proof}

We now turn to ignoring terms of degree higher than the dominant degree of \eqref{eqn}.
First, some notation.
Recall that for $F \in K\{Y\}$, we set $F_{\les n} \coloneqq F_0+F_1+\dots+F_n$.
Note that if $n \ges \ddeg F$, then $D_F=D_{F_{\les n}}$.
Now set $F \coloneqq P_{+f-g}$, so $d \ges \ddeg F_{\xf m}$ for all $\f m \prece g$.
Consider the ``truncation''
\begin{equation}\tag{E${}_{g, \les d}$}\label{eqnglesd}
F_{\les d}(Y)\ =\ 0, \hs Y\prece g
\end{equation}
of \eqref{eqng} as an asymptotic differential equation over $K$.
We have, for all $\f m \prece g$,
\[
D_{F_{\xf m}}\ =\ D_{(F_{\xf m})_{\les d}}\ =\ D_{(F_{\les d})_{\xf m}},
\]
so \eqref{eqnglesd} has the same algebraic starting monomials and dominant degree as \eqref{eqng}.
Next, we show that under suitable conditions the unraveller $(g, \ca E'_g)$ for \eqref{eqng} from the previous lemma remains an unraveller for \eqref{eqnglesd}.
Recall that $[\gamma]$ denotes the archimedean class of $\gamma \in \Gamma$ and that such classes are ordered in the natural way; see \S\ref{prelim:arch}.

The next lemma corresponds to \cite[Lemma~13.8.13]{adamtt}.
The essential difference is that we use the valuation $vg \mapsto [vg]$ on $\Gamma$ instead of the valuation $vg \mapsto v(g'/g)$ used in \cite[Lemma~13.8.13]{adamtt}.
\begin{lem}\label{reducedegunravel}
Suppose that $[v(\f e/g)] < [v(g/f)]$.
Then $(g, \ca E'_g)$ is an unraveller for \eqref{eqnglesd}, and $\f e$ is the largest algebraic starting monomial for the unravelled asymptotic differential equation
\begin{equation}\tag{E${}_{g, \les d}'$}\label{eqnglesd'}
(F_{\les d})_{+g}(Y)\ =\ 0, \hs Y\in \ca E'_g.
\end{equation}
\end{lem}
\begin{proof}
First, we reduce to the case $g \asymp 1$: set $\f g \coloneqq \f d_g$ and replace $P$, $f$, $g$, $\ca E$, and $\ca E'$ by $P_{\xf g}$, $f/\f g$, $g/\f g$, $\f g^{-1}\ca E$, and $\f g^{-1}\ca E'$, respectively, and use Lemma~\ref{multconjade}.
Note that now $\f e \prec 1 \prec f$ and $[v\f e]<[vf]$.

Since $F = P_{+f-g}$ and $g \asymp 1$, we have $\ddeg F = \ddeg_{\prece 1} F = d$ by Lemma~\ref{tildeunravel}, so
\[
d\ \les\ \ddeg F_{\x f}\ \les\ \ddeg_{\ca E} F\ =\ d,
\]
using Corollary~\ref{adh6.6.7} and Lemma~\ref{adh6.6.10}.
This yields $d = \ddeg F = \ddeg F_{\x f}$.
For $\f m$ with $[v\f m]<[vf]$ we may thus apply Corollary~\ref{reducedeg} with $F$ and $\f d_f$ in place of $P$ and $\f n$ to get
\begin{equation}\label{reducedegunravel:dompart}
D_{P_{+f, \xf m}}\ =\ D_{F_{+g, \xf m}}\ =\ D_{(F_{\les d})_{+g, \xf m}}.
\end{equation}
In particular, this holds if $\f e \prece \f m \prec 1$, as then $[v\f m]\les[v\f e]<[vf]$.
Thus $\f e$ is the largest algebraic starting monomial for \eqref{eqnglesd'}, since it is the largest such for \eqref{eqn'2}.

For $(g, \ca E'_g)$ to be an unraveller for \eqref{eqnglesd}, we now show:
\begin{enumerate}
	\item $\ddeg_{\ca E'_g} (F_{\les d})_{+g}=d$;
	\item $\ddeg_{\prec h} (F_{\les d})_{+g+h}<d$ for all $h \in \ca E'_g$.
\end{enumerate}
For (i), if $\f{e\prece m} \in \ca E'_g$, then by Corollary~\ref{ndiagcor} and \eqref{reducedegunravel:dompart} we have
\[
d\ =\ \ddeg P_{+f, \xf m}\ =\ \ddeg (F_{\les d})_{+g, \xf m}.
\]
For (ii), let $h \in \ca E'_g$, so $h \in \ca E'$ and $h \prec 1$.
Set $\f h \coloneqq \f d_h$ and $u \coloneqq h/\f h$.
Applying Lemma~\ref{ddegtrans}, we have
\begin{align}
\ddeg_{\prec h} (F_{\les d})_{+g+h}\ &=\ \mul\big(D_{(F_{\les d})_{+g, \xf h}}\big)_{+\bar u};\label{reducedegunravel:trans1}\\
\ddeg_{\prec h} P_{+f+h}\ &=\ \mul\big(D_{P_{+f, \xf h}}\big)_{+\bar u}.\label{reducedegunravel:trans2}
\end{align}
First suppose $\f e \prece h$, so then combining \eqref{reducedegunravel:dompart}, for $\f m = \f h$, with \eqref{reducedegunravel:trans1} and \eqref{reducedegunravel:trans2} we have
\[
\ddeg_{\prec h} (F_{\les d})_{+g+h}\ =\ \ddeg_{\prec h}P_{+f+h}\ <\ d,
\]
since \eqref{eqn'2} is unravelled.
Now suppose $h \prec \f e$.
If $\f e^2 \prece h \prec \f e$, then $[vh]=[v\f e]<[vf]$, and thus by \eqref{reducedegunravel:dompart} and Corollary~\ref{ndiagcor},
\[
\ddeg (F_{\les d})_{+g, \xf h}\ =\ \ddeg P_{+f, \xf h}\ <\ \ddeg P_{+f, \xf e}\ =\ d.
\]
By Corollary~\ref{adh6.6.7}, $\ddeg (F_{\les d})_{+g, \xf h}<d$ remains true for any $h \prec \f e$.
Hence, by \eqref{reducedegunravel:trans1},
\[
\ddeg_{\prec h} (F_{\les d})_{+g+h}\ =\ \mul\big(D_{(F_{\les d})_{+g, \xf h}}\big)_{+\bar u}\ \les\ \ddeg(F_{\les d})_{+g, \xf h}\ <\ d,
\]
which completes the proof of (ii).
\end{proof}

\subsection{Finding solutions in differential-henselian fields}\label{sec:ade:dhsolns}

We now use $\d$-henselianity to find solutions of asymptotic differential equations.
Given an element of an extension of $K$, when $K$ has few constants we find a solution closest to that element.
The only result in this subsection that uses the assumption that $\Gamma$ has no least positive element is Lemma~\ref{bestapproxsoln}.

We say that \eqref{eqn} is \emph{quasilinear} if $\ddeg_{\ca E} P=1$.
Note that if $K$ is $r$-$\d$-henselian and \eqref{eqn} is quasilinear, then $P$ has a zero in $\ca E \cup \{0\}$.
Note that even in this case, \eqref{eqn} may not have a solution, since those are required to be nonzero.
The next lemma, routinely adapted from \cite[Lemma~14.3.4]{adamtt}, shows how certain approximate solutions yield solutions.

\begin{lem}\label{approxtoreal}
Suppose that $K$ is $r$-$\d$-henselian.
Let $g \in K^{\x}$ be an approximate zero of $P$ such that $\ddeg P_{\x g} = 1$.
Then there exists $y \sim g$ in $K$ such that $P(y)=0$.
\end{lem}
\begin{proof}
Let $\f m \coloneqq \f d_g$ and $u \coloneqq g/\f m$, so $D_{P_{\xf m}}(\bar{u})=0$ and thus
\[
\dmul P_{\xf m, +u}\ =\ \mul (D_{P_{\xf m}})_{+\bar{u}}\ \ges\ 1.
\]
By Lemma~\ref{adh6.6.5}, we also have
\[
\dmul P_{\xf m, +u}\ \les\ \ddeg P_{\xf m, +u}\ =\ \ddeg P_{\xf m}\ =\ 1.
\]
Thus $\dmul P_{\xf m, +u}=1$, so by $r$-$\d$-henselianity we have $z \prec 1$ with $P_{\xf m, +u}(z)=0$.
Setting $y \coloneqq (u+z)\f m$ gives $P(y)=0$ and $y \sim u\f m=g$.
\end{proof}

Now let $f$ be an element of an extension of $K$.
We say that a solution $y$ of \eqref{eqn} \emph{best approximates $f$} (\emph{among solutions of \eqref{eqn}}) if $y-f \prece z-f$ for each solution $z$ of \eqref{eqn}.
Note that if $f \in K^\x$ is a solution of \eqref{eqn}, then $f$ is the unique solution of \eqref{eqn} that best approximates $f$.
Also, if $f \succ \ca E$, then $y-f \asymp f$ for all $y \in \ca E$, and so each solution of \eqref{eqn} best approximates $f$.
The next lemma concerning multiplicative conjugation has a routine proof, identical to that of \cite[Lemma~11.2.10]{adamtt}, by distinguishing the cases $z \succ g$ and $z \prece g$.
\begin{lem}\label{aeconj}
Let $f$ be an element of an extension of $K$ and let $g \in \ca E$ with $f \prece g$. Suppose that $y$ is a solution of the asymptotic differential equation
\[
P_{\x g}(Y)\ =\ 0, \hs Y \prece 1
\]
that best approximates $g^{-1}f$.
Then the solution $gy$ of \eqref{eqn} best approximates $f$.
\end{lem}
%\begin{proof}
%Let $z$ be a solution of \eqref{eqn}.
%If $z \succ g \succe f$, then $z-f \sim z$.
%As $y \prece 1$ and $f \prece g$, we have $gy - f \prece g$.
%Combining these two yields $gy-f \prece z-f$.
%If $z \prece g$, then $g^{-1}z \prece 1$ is a solution of the above asymptotic differential equation and so by assumption on $y$, we have $y-g^{-1}f \prece g^{-1}z-g^{-1}f$,
%and hence $gy-f \prece z-f$.
%\end{proof}

This next lemma is easily adapted from \cite[Lemma~14.1.13]{adamtt}; it is the one place in this section that we impose the assumption $C \subseteq \ca O$.
\begin{lem}\label{bestapprox}
Suppose that $r \ges 1$ and $K$ is $r$-$\d$-henselian with $C \subseteq \ca O$.
Suppose that \eqref{eqn} is quasilinear and has a solution.
Let $f$ be an element of an extension of $K$.
Then $f$ is best approximated by some solution of \eqref{eqn}.
\end{lem}
\begin{proof}
By the comment above Lemma~\ref{aeconj}, we may assume that $f \not\succ \ca E$.
Thus we may take $g \in \ca E$ with $f \prece g$ such that \eqref{eqn} has a solution $y \prece g$ and
\[
\ddeg P_{\x g}\ =\ \ddeg_{\ca E} P\ =\ 1.
\]
By Lemma~\ref{aeconj}, we may replace $P$ by $P_{\x g}$ and $\ca E$ by $\ca O^{\neq}$ in order to assume that $\ca E = \ca O^{\neq}$.
Suppose that $f$ is not best approximated by any solution of \eqref{eqn}.
Then for each $i$ we get $y_i \in K^{\x}$ such that:
\begin{enumerate}
	\item $y_i$ is a solution of \eqref{eqn}, i.e., $P(y_i)=0$ and $y_i \prece 1$;
	\item $y_i-f \succ y_{i+1}-f$;
	\item $\ddeg P_{+y_i}=\ddeg P=1$ (by Lemma~\ref{adh6.6.5}).
\end{enumerate}
Item (ii) implies that $y_{i+1}-y_i \asymp y_i-f$, contradicting \cite[Lemma~7.5.5]{adamtt}.
\end{proof}

The next lemma is based on \cite[Lemma~14.1.14]{adamtt} but has a shorter proof using Lemma~\ref{approxtoreal}.
\begin{lem}\label{approxtorealade}
Suppose that $K$ is $r$-$\d$-henselian, \eqref{eqn} is quasilinear, and $f \in \ca E$ is an approximate solution of \eqref{eqn}.
Then \eqref{eqn} has a solution $y_0 \sim f$, and every solution $y$ of \eqref{eqn} that best approximates $f$ satisfies $y \sim f$.
\end{lem}
\begin{proof}
Let $\f m \coloneqq \f d_f$ and $u \coloneqq f/\f m$.
Then by Lemma~\ref{adh6.6.5} we have
\[
\ddeg P_{\x f}\ =\ \ddeg P_{\x \f{m}, + u}\ \ges\ \dmul P_{\x \f{m}, + u}\ \ges\ 1,
\]
and so since \eqref{eqn} is quasilinear, $\ddeg P_{\x f}=1$.
Thus Lemma~\ref{approxtoreal} yields a solution $y_0 \sim f$ of \eqref{eqn}.
If $y$ is a solution of \eqref{eqn} that best approximates $f$, then $y \sim f$, as
\[
y-f\ \prece\ y_0-f\ \prec\ f.
\qedhere
\]
\end{proof}

For the next lemma, a routine adaptation of  \cite[Lemma~14.3.13]{adamtt}, recall from \S\ref{secade:unravelimmed} that given an immediate extension $L$ of $K$, we extend the asymptotic differential equation \eqref{eqn} over $K$ to \eqref{eqnL} over $L$.
Note that if \eqref{eqn} is quasilinear, then so is \eqref{eqnL}.

\begin{lem}\label{bestapproxsoln}
Suppose that $K$ is $r$-$\d$-henselian and let $L$ be an immediate extension of $K$.
Suppose that \eqref{eqn} is quasilinear, $\ca E' \subseteq \ca E$ is $\prece$-closed, and $f \in \ca E_L$ is such that the refinement
\begin{equation}\tag{E${}'_L$}\label{eqn'L}
P_{+f}(Y)\ =\ 0, \hs Y \in \ca E'_L
\end{equation}
of \eqref{eqnL} is also quasilinear.
Let $y \prece f$ be a solution of \eqref{eqn} that best approximates $f$.
Then $f-y \in \ca E'_L \cup \{0\}$.
\end{lem}
\begin{proof}
The case $f=y$ being trivial, suppose that $f \neq y$ and set $\f m \coloneqq \f d_{f-y}$.
As $f-y \in \ca E_L$, we have $\f m \in \ca E$.
Now suppose towards a contradiction that $f-y \notin \ca E'_L$.
Then $\ca E'_L \prec \f m \in \ca E$, so by quasilinearity and Lemma~\ref{adh6.6.10},
\[
1\ =\ \ddeg_{\ca E'_L} P_{+f}\ \les\ \ddeg_{\prece \f m} P_{+f}\ =\ \ddeg_{\prece \f m} P_{+y}\ \les\ \ddeg_{\ca E} P_{+y}\ =\ \ddeg_{\ca E} P\ =\ 1.
\]
Hence the asymptotic differential equation
\begin{equation}\label{eqnym}
P_{+y}(Y)\ =\ 0, \hs Y \prece \f m
\end{equation}
over $K$ is also quasilinear.
Also, by the quasilinearity of \eqref{eqn'L}, we have
\[
\ddeg_{\prec \f m} (P_{+y})_{+(f-y)}\ =\ \ddeg_{\prec \f m} P_{+f}\ \ges\ \ddeg_{\ca E'_L} P_{+f}\ =\ 1,
\]
so $f-y$ is an approximate solution of \eqref{eqnym} over $L$, by Corollary~\ref{approxsolnequiv}.
Take $g \in K^{\x}$ with $g \sim f-y$, so $g$ is an approximate solution of \eqref{eqnym} over $K$, and, by the quasilinearity of \eqref{eqnym},
\[
\ddeg P_{+y, \x g}\ =\ \ddeg_{\prece \f m} P_{+y}\ =\ 1.
\]
Then by Lemma~\ref{approxtoreal} we have $z \sim g \sim f-y$ in $K$ such that $P(y+z)=0$.
We must have $y+z \neq 0$, as otherwise $f \prec y-f$, contradicting $y \prece f$.
From $y \prece f$, we also obtain $y+z \prece f$, so $y+z \in \ca E$.
Since $y+z-f \prec y-f$, this contradicts that $y$ best approximates $f$.
\end{proof}

\section{Reducing complexity}\label{sec:redcomplex}

This is a technical section whose main goal is Proposition~\ref{maintechprop}.
This proposition, or rather its consequence Corollary~\ref{maintechcor}, is the linchpin of Proposition~\ref{mainlemmadiv}, and its proof uses all of the previous sections and some additional results from \cite{adamtt}.
This section is based on \cite[\S14.4]{adamtt}.

\begin{ass}
In this section, $K$ is asymptotic and has a monomial group $\f M$, $\Gamma$ is divisible, and $\bm k$ is $r$-linearly surjective with $r \ges 1$.
\end{ass}

Let $\f m$ and $\f n$ range over $\f M$.
As usual, we let $P \in K\{Y\}^{\neq}$ with order at most $r$.
As in the previous section, let $\ca E \subseteq K^{\x}$ be $\prece$-closed, so we have an asymptotic differential equation
\begin{equation}\tag{E}\label{neweqn}
P(Y)\ =\ 0, \hs Y \in \ca E
\end{equation}
over $K$.
Set $d \coloneqq \ddeg_{\ca E} P$ and suppose that $d \ges 2$.
We fix an immediate asymptotic $r$-$\d$-henselian extension $\wh{K}$ of $K$ and use $\f M$ as a monomial group of $\wh{K}$.

Let $\wh{\ca E} \coloneqq \ca E_{\wh{K}}= \{y \in \wh{K}^{\x} : vy \in v\ca E\}$, so we have the asymptotic differential equation
\begin{equation}\tag{\^{E}}\label{wheqn}
P(Y)\ =\ 0, \hs Y \in \wh{\ca E}
\end{equation}
over $\wh{K}$ with dominant degree $d$.
Suppose that \eqref{wheqn} is not unravelled, and that this is witnessed by an $\wh{f} \in \wh{\ca E}$ such that $(\wh{f}, \wh{\ca E}')$ is an unraveller for \eqref{wheqn}.
That is, $\ddeg_{\prec \wh f} P_{+\wh f}=d$, and the refinement
\begin{equation}\tag{\^{E}${}'$}\label{wheqn'}
P_{+\wh{f}}(Y)\ =\ 0, \hs Y \in \wh{\ca E}'
\end{equation}
of \eqref{wheqn} is unravelled with dominant degree $d$.
By Corollary~\ref{approxsolnequiv}, $\wh f$ is an approximate solution of \eqref{wheqn} of multiplicity $d$.
Note also that $\wh{\ca E}' = \ca E'_{\wh{K}}$ for the $\prece$-closed set $\ca E' \coloneqq \wh{\ca E}' \cap K \subseteq \ca E$.
Since \eqref{wheqn} is not unravelled, neither is \eqref{eqn} by the discussion preceding Lemma~\ref{unravelimmedlem}.
Suppose also that $\mul P_{+\wh f}<d$, so that by Proposition~\ref{ndiag}, \eqref{wheqn'} has an algebraic starting monomial; let $\f e$ be the largest such.

The main proposition corresponds to \cite[Proposition~14.4.1]{adamtt}.
\begin{prop}\label{maintechprop}
There exists $f \in \wh K$ such that one of the following holds:
\begin{enumerate}
\item $\wh f-f \prece \f e$ and $A(f)=0$ for some $A \in K\{Y\}$ with $c(A)<c(P)$ and $\deg A=1$;
\item $\wh f \sim f$, $\wh f-a \prece f-a$ for all $a \in K$, and $A(f)=0$ for some $A \in K\{Y\}$ with $c(A)<c(P)$ and $\ddeg A_{\x f}=1$.
\end{enumerate}
\end{prop}

\subsection{Special case}\label{unravelspecial}
We first prove Proposition~\ref{maintechprop} in the special case that $\ddeg_{\ca E} P = \deg P$ and later reduce to this case using Lemma~\ref{slowdownreducedeg}.
Below, we consider the $\d$-polynomial
\[
P_{+\wh{f}, \xf{e}} \in \wh{K}\{Y\};
\]
note that $\ddeg P_{+\wh{f}, \xf{e}} = d$ by the choice of $\f e$.
Let $s \les r$ be the order of $P$.
For $\bm i \in \N^{1+s}$, we let
\[
\partial^{\bm i}\ \coloneqq\ \frac{\partial^{|\bm i|}}{\partial Y^{i_0} \dots \partial \big(Y^{(s)}\big)^{i_s}}
\]
denote the partial differential operator on $\wh{K}\{Y\}$ that differentiates $i_n$ times with respect to $Y^{(n)}$ for $n=0,\dots,s$.
(We also use additive and multiplicative conjugates of partial differential operators; see \cite[\S12.8]{adamtt}.)
For any partial differential operator (in the sense of \cite[\S12.7]{adamtt}) $\Delta$ on $\wh{K}\{Y\}$, any $Q \in \wh{K}\{Y\}$, and any $a \in \wh{K}$,
\[
\Delta (Q_{+a})\ =\ (\Delta Q)_{+a}
\]
by \cite[Lemma~12.8.7]{adamtt}, so we write $\Delta Q_{+a}$ and do not distinguish between these.
If $a \in \wh{K}^{\x}$, note that, by \cite[Lemma~12.8.8]{adamtt},
\[
\Delta Q_{\x a}\ \coloneqq\ \Delta (Q_{\x a})\ =\ (\Delta_{\x a} Q)_{\x a},
\]
Note that, when no parentheses are used, we intend additive and multiplicative conjugation of $Q$ to take place before $\Delta$ is applied, in order to simplify notation.

Now, choose $\bm i \in \N^{1+s}$ such that $\deg(\partial^{\bm i} Y^{\bm j})=1$ for some $\bm j \in \N^{1+s}$ with $|\bm j|=d$ and
\[
\big(P_{+\wh{f}, \xf{e}}\big)_{\bm j}\ \asymp\ P_{+\wh{f}, \xf{e}}.
\]
In particular, $|\bm i|=d-1$ and
\[
\ddeg \partial^{\bm i} P_{+\wh{f}, \xf{e}}\ =\ \deg D_{\partial^{\bm i} P_{+\wh{f}, \xf{e}}}\ =\ \deg \partial^{\bm i} D_{P_{+\wh{f}, \xf{e}}}\ =\ 1.
\]
We consider the partial differential operator $\Delta \coloneqq (\partial^{\bm i})_{\xf e}$ on $\wh{K}\{Y\}$.
We have
\[
(\Delta P)_{+\wh{f}, \xf{e}}\ =\ \partial^{\bm i} P_{+\wh{f}, \xf{e}}
\]
by \cite[Lemmas~12.8.7 and 12.8.8]{adamtt}.
Hence the asymptotic differential equation
\[
\Delta P_{+\wh{f}}(Y)\ =\ 0, \hs Y \prece \f{e}
\]
is quasilinear.
In \cite[\S14.4]{adamtt}, partial differentiation is also used to obtain a quasilinear asymptotic differential equation.
Under the powerful assumption of $\upomega$-freeness made in that setting, newton polynomials have a very special form, and so a specific choice of $\Delta$ was needed.
Here, and in the next subsection, a similar technique works despite the lack of restrictions on dominant parts.

This next lemma is routinely adapted from \cite[Lemma~14.4.2]{adamtt}.
\begin{lem}\label{solnexists}
Suppose that $\f{e} \prec \wh{f}$ and the asymptotic differential equation
\begin{equation}\label{eqnDelta}
\Delta P(Y)\ =\ 0, \hs Y \in \wh{\ca E}
\end{equation}
over $\wh{K}$ is quasilinear.
Then \eqref{eqnDelta} has a solution $y \sim \wh{f}$, and if $f$ is any solution of \eqref{eqnDelta} that best approximates $\wh f$, then $f-\wh f \prece \f e$.
\end{lem}
\begin{proof}
By Lemma~\ref{adh6.6.10}, we have
\[
\ddeg_{\prec \wh f} \Delta P_{+\wh f}\ \les\ \ddeg_{\wh{\ca E}} \Delta P_{+\wh f}\ =\ \ddeg_{\wh{\ca E}} \Delta P\ =\ 1.
\]
But from $\f{e} \prec \wh{f}$, we also have
\[
1\ =\ \ddeg_{\prece \f e} \Delta P_{+\wh{f}}\ \les\ \ddeg_{\prec \wh{f}} \Delta P_{+\wh{f}},
\]
so $\ddeg_{\prec \wh f} \Delta P_{+\wh f}=1$.
Then since $\wh K$ is $r$-$\d$-henselian, we get $y \sim \wh f$ with $\Delta P(y)=0$.

For the second statement, the refinement
\begin{equation}\label{eqnDeltaref}
\Delta P_{+\wh{f}}(Y)\ =\ 0, \hs Y \prece \f{e}
\end{equation}
of \eqref{eqnDelta} is quasilinear, so we can apply Lemma~\ref{bestapproxsoln} with $\wh{K}$ in the roles of both $L$ and $K$, and $\Delta P$, $\wh f$, $f$, \eqref{eqnDelta}, and \eqref{eqnDeltaref} in the roles of $P$, $f$, $y$, \eqref{eqn}, and \eqref{eqn'L}, respectively.
\end{proof}

We now conclude the proof of Proposition~\ref{maintechprop} in the case that $\deg P = d$, easily adapted from \cite[Corollary~14.4.3]{adamtt}.
Recall that $d \ges 2$.

\begin{lem}\label{maintechpropspecial}
Suppose that $\deg P = d$.
Then there exist $f \in \wh{K}$ and $A \in K\{Y\}$ such that $\wh{f}-f \prece \f{e}$, $A(f)=0$, $c(A)<c(P)$, and $\deg A = 1$.
\end{lem}
\begin{proof}
Since $\deg P = d$, we also have $\deg P_{+\wh{f}, \xf{e}} = d$, and hence
\[
\deg \Delta P\ =\ \deg (\Delta P)_{+\wh{f}, \xf{e}}\ =\ \deg \partial^{\bm i} P_{+\wh{f}, \xf{e}}\ =\ 1,
\]
by the choice of $\bm i$.
Hence \eqref{eqnDelta} is quasilinear.

If $\wh{f} \prece \f{e}$, then $f \coloneqq 0$ and $A \coloneqq Y$ work, so assume that $\f{e} \prec \wh{f}$.
First, Lemma~\ref{solnexists} yields a solution $y \sim \wh f$ of \eqref{eqnDelta}.
As $\wh K$ has few constants, Lemma \ref{bestapprox} gives that $\wh{f}$ is best approximated by some solution $f$ of \eqref{eqnDelta}.
So applying Lemma~\ref{solnexists} again, we have $f-\wh{f} \prece \f{e}$.
Then $\Delta P(f)=0$, $c(\Delta P)<c(P)$, and $\deg \Delta P = 1$, so we may take $A \coloneqq \Delta P$.
\end{proof}

\subsection{Tschirnhaus refinements}\label{sec:redcomplex:tschirn}
Set $\f{f \coloneqq d}_{\wh{f}}$, and we now consider the $\d$-polynomial $P_{\xf f} \in K\{Y\}^{\neq}$.
If $\f e \succe \f f$, then the first case of Proposition~\ref{maintechprop} holds for $f \coloneqq 0$ and $A \coloneqq Y$, so in the rest of this subsection and in \S\ref{sec:redcomplex:slowdown} and \S\ref{sec:redcomplex:slowdowncor} we suppose that $\f{e \prec f}$.
Then we have, by the choice of $\f e$ and Lemma~\ref{adh6.6.10},
\[
d\ =\ \ddeg_{\prece \f e}P_{+\wh{f}}\ \les\ \ddeg_{\prece \f{f}} P_{+\wh{f}}\ =\ \ddeg_{\prece \f{f}} P\ \les\ \ddeg_{\wh{\ca E}} P\ =\ d,
\]
and thus $\ddeg P_{\xf f}=d$.

Now, choose $\bm i \in \N^{1+s}$ so that $\deg (\partial^{\bm i} Y^{\bm j})$=1 for some $\bm j \in \N^{1+s}$ with $|\bm j|=d$ and $(P_{\xf f})_{\bm j} \asymp P_{\xf f}$.
Thus we have $|\bm i|=d-1$ and
\[
D_{\partial^{\bm i} P_{\xf f}}\ =\ \partial^{\bm i} D_{P_{\xf f}},
\]
and so $\ddeg \partial^{\bm i} P_{\xf f} = 1$.
We consider the partial differential operator $\Delta \coloneqq (\partial^{\bm i})_{\xf f}$ on $\wh{K}\{Y\}$.
By \cite[Lemma~12.8.8]{adamtt},
\[
\left( \Delta P \right)_{\xf f}\ =\ \partial^{\bm i} P_{\xf f},
\]
and thus the asymptotic differential equation
\begin{equation}\label{eqnDeltaf}
\Delta P(Y)\ =\ 0, \hs Y \prece \f f
\end{equation}
over $\wh K$ is quasilinear.
The comments from \S\ref{unravelspecial} about the difference between the partial differentiation used here and that used in \cite[\S14.4]{adamtt} apply in this subsection as well, and necessitate a slight weakening of the following lemma from its counterpart \cite[Lemma~14.4.4]{adamtt}.
It follows immediately from the quasilinearity of \eqref{eqnDeltaf} by Corollary~\ref{adh6.6.11}.
\begin{lem}\label{noasm}
Suppose that $f \in \wh K$ is a solution of \eqref{eqnDeltaf}.
Then for all $g \in \wh{K}^{\x}$ with $g \prece \f f$ we have
\[
\mul \left(\Delta P\right)_{+f, \x g}\ =\ \ddeg \left(\Delta P\right)_{+f, \x g}\ =\ 1,
\]
and hence $(\Delta P)_{+f}$ has no algebraic starting monomial $\f g \in \f M$ with $\f g \prece \f f$.
\end{lem}
%\begin{proof}
%By Corollary~\ref{adh6.6.11}, we have for all $\f{g \prece f}$,
%\[
%\mul (\Delta P)_{+f, \xf g}\ =\ \dmul (\Delta P)_{+f, \xf g}\ =\ \ddeg (\Delta P)_{+f, \xf g}\ =\ 1.
%\]
%If $(\Delta P)_{+f}$ were to have an algebraic starting monomial $\f g \prece \f f$, then by Proposition~\ref{ndiag},
%\[
%1\ =\ \dmul (\Delta P)_{+f, \xf g}\ <\ \ddeg (\Delta P)_{+f, \xf g}\ =\ 1,
%\]
%a contradiction.
%\end{proof}

The next statement is based on \cite[Lemma~14.4.5]{adamtt} but has a different proof due to the different choices of $\Delta$ needed here and in \cite[\S14.4]{adamtt}.
\begin{lem}\label{tschirnapproxsoln}
The element $\wh f \in \wh K$ is an approximate solution of \eqref{eqnDeltaf}.
\end{lem}
\begin{proof}
Set $u \coloneqq \wh{f}/\f f$.
Since $\wh f$ is an approximate zero of $P$ of multiplicity $d = \ddeg P_{\xf f} = \deg D_{P_{\xf f}}$,
\[
\big(D_{P_{\xf f}}\big)_{+\bar{u}}\ =\ \sum_{|\bm j|=d}\big(D_{P_{\xf f}}\big)_{\bm j} Y^{\bm j},
\]
by \cite[Lemma~4.3.1]{adamtt}, where $\bm j$ ranges over $\N^{1+s}$.
Then
\[
\big( \partial^{\bm i} D_{P_{\xf f}} \big)_{+\bar{u}}\
=\ \partial^{\bm i} \big( D_{P_{\xf f}} \big)_{+\bar{u}}\
=\ \sum_{|\bm j|=d} \big( D_{P_{\xf f}} \big)_{\bm j} \partial^{\bm i} Y^{\bm j},
\]
so the multiplicity of $\partial^{\bm i} D_{P_{\xf f}}$ at $\bar u$ is $1$ by the choice of $\bm i$.
In view of
\[
D_{(\Delta P)_{\xf f}}\ =\ D_{\partial^{\bm i} P_{\xf f}}\ =\ \partial^{\bm i} D_{P_{\xf f}},
\]
$\wh{f}$ is an approximate solution of \eqref{eqnDeltaf}.
\end{proof}

Let $f \in \wh K$ with $f \sim \wh f$, so $\ddeg_{\prec \f f} P_{+f} = \ddeg_{\prec \f f} P_{+\wh f} = d$ by Lemma~\ref{adh6.6.10}.
That is, the refinement
\begin{equation}\tag{T}\label{eqnT}
P_{+f}(Y)\ =\ 0, \hs Y \prec \f f
\end{equation}
of \eqref{wheqn} still has dominant degree $d$.
As $\wh f$ is an approximate solution of \eqref{eqnDeltaf}, Lemmas~\ref{approxtorealade} and~\ref{bestapprox} give a solution $f_0 \in \wh K$ of \eqref{eqnDeltaf} that best approximates $\wh f$ with $f_0 \sim \wh f \sim f$.
Thus
\[
\ddeg_{\prec \f f} \Delta P_{+f}\ =\ \ddeg_{\prec \f f} \Delta P_{+f_0}\ =\ 1
\]
by Lemmas~\ref{adh6.6.10} and \ref{noasm}.
Hence the refinement
\begin{equation}\tag{$\Delta$T}\label{eqnDeltaT}
\Delta P_{+f}(Y)\ =\ 0, \hs Y \prec \f f
\end{equation}
of \eqref{eqnDeltaf} is also quasilinear.

\begin{defn}
A \emph{Tschirnhaus refinement of \eqref{wheqn}} is an asymptotic differential equation \eqref{eqnT} over $\wh K$ as above with $\wh f \sim f \in \wh K$ such that some solution $f_0 \in \wh K$ of \eqref{eqnDeltaf} over $\wh K$ best approximates $\wh f$ and satisfies $f_0-\wh f \sim f-\wh f$.
\end{defn}

\begin{defn}
Let $f, \wh g \in \wh K$ and $\f m$ satisfy
\[
\f m\ \prec\ f-\wh f\ \prece\ \wh g\ \prec\ \f f,
\]
so in particular $f \sim \wh f$.
With \eqref{eqnT} as above, but not necessarily a Tschirnhaus refinement of \eqref{wheqn}, we say that the refinement
\begin{equation}\tag{TC}\label{eqnTC}
P_{+f+\wh g}(Y)\ =\ 0, \hs Y \prece \f m
\end{equation}
of \eqref{eqnT} is \emph{compatible with \eqref{eqnT}} if it has dominant degree $d$ and $\wh g$ is not an approximate solution of \eqref{eqnDeltaT}.
\end{defn}

The next two lemmas are routine adaptations of \cite[Lemmas~14.4.7 and 14.4.8]{adamtt}.
\begin{lem}
Let $f, f_0, \wh g \in \wh K$ and $\f m$ be such that
\[
\f m\ \prec\ f_0-\wh f\ \sim\ f-\wh f\ \prece\ \wh g\ \prec\ \f f,
\]
and \eqref{eqnTC} has dominant degree $d$.
Then $\wh g$ is an approximate solution of \eqref{eqnT} and of
\begin{equation}\tag{T${}_0$}\label{eqnT0}
P_{+f_0}(Y)\ =\ 0, \hs Y \prec \f f.
\end{equation}
\end{lem}
\begin{proof}
First, $\wh g$ is an approximate solution of \eqref{eqnT} by Lemma~\ref{approxsolncheck}, since $\f m \prec \wh g$ and
\[
\ddeg_{\prece \f m} P_{+f+\wh g}\ =\ d\ =\ \ddeg_{\prec \f f}P_{+f}.
\]
From $f_0-f \prec f- \wh f \prece \wh g$ and $\f m \prec \wh g$, we obtain, using Lemma~\ref{adh6.6.10} in the first equality,
\[
\ddeg_{\prec\wh g} P_{+f_0+\wh g}\ =\ \ddeg_{\prec \wh g} P_{+f+\wh g}\ \ges\ \ddeg_{\prece \f m} P_{+f+\wh g}\ =\ d\ \ges\ 1,
\]
so $\wh g$ is an approximate solution of \eqref{eqnT0} by Corollary~\ref{approxsolnequiv}.
\end{proof}

\begin{lem}\label{approxsolnT0}
Let $f, f_0, \wh g \in \wh K$ with
\[
f_0-\wh f\ \sim\ f-\wh f\ \prece\ \wh g\ \prec\ \f f.
\]
Then $\wh g$ is an approximate solution of \eqref{eqnDeltaT} if and only if $\wh g$ is an approximate solution of
\begin{equation}\tag{$\Delta$T${}_0$}\label{eqnDeltaT0}
\Delta P_{+f_0}(Y)\ =\ 0, \hs Y \prec \f f.
\end{equation}
\end{lem}
\begin{proof}
Again, since $f_0-f \prec \wh f-f \prece \wh g$, by Lemma~\ref{adh6.6.10} we have
\[
\ddeg_{\prec \wh g} \Delta P_{+f_0+\wh g}\ =\ \ddeg_{\prec\wh g} \Delta P_{+f+\wh g}.
\]
The result then follows from Corollary~\ref{approxsolnequiv}, since $\wh g \prec \f f$.
\end{proof}

Note that, for any $f_0 \sim f$, the equation \eqref{eqnDeltaT0} in the previous lemma is quasilinear by Lemma~\ref{adh6.6.10}, since \eqref{eqnDeltaT} is.
The next lemma gives compatible refinements of \eqref{eqnT} when $\f e \prec f-\wh f$ in the same way as \cite[Lemma~14.4.9]{adamtt}.
\begin{lem}\label{getcompatibleref}
Suppose that \eqref{eqnT} is a Tschirnhaus refinement of \eqref{wheqn} and $\f e\prec f-\wh f$.
Then, with $\wh g \coloneqq \wh f-f$ and $\f m \coloneqq \f e$, the refinement \eqref{eqnTC} of \eqref{eqnT} is compatible with \eqref{eqnT}.
\end{lem}
\begin{proof}
Since $\f e$ is the largest algebraic starting monomial for \eqref{wheqn'}, 
\[
\ddeg_{\prece\f e} P_{+f+\wh g}\ =\ \ddeg_{\prece \f e} P_{+\wh f}\ =\ \ddeg P_{+\wh f, \xf e}\ =\ d,
\]
and so \eqref{eqnTC} has dominant degree $d$.

As \eqref{eqnT} is a Tschirnhaus refinement of \eqref{wheqn}, let $f_0 \in \wh K$ be a solution of \eqref{eqnDeltaf} that best approximates $\wh f$ and satisfies $f-\wh f \sim f_0-\wh f$.
Suppose towards a contradiction that $\wh g$ is an approximate solution of \eqref{eqnDeltaT}, so by Lemma~\ref{approxsolnT0}, $\wh g$ is also an approximate solution of \eqref{eqnDeltaT0}.
Then by Lemma~\ref{approxtorealade}, \eqref{eqnDeltaT0} has a solution $y \sim \wh g \sim \wh f-f_0$.
Thus $\Delta P(f_0+y)=0$, so $f_0+y$ is a solution of \eqref{eqnDeltaf}, since $f_0+y \prece \f f$.
But also
\[
f_0 + y - \wh f\ =\ y - (\wh f - f_0)\ \prec\ \wh f - f_0,
\]
contradicting that $f_0$ best approximates $\wh f$.
Hence $\wh g$ is not an approximate solution of \eqref{eqnDeltaT}, and so \eqref{eqnTC} is compatible with \eqref{eqnT}.
\end{proof}
In fact, the proof above shows that \eqref{eqnDeltaT0} has no approximate solution $h$ with $h \sim \wh f-f_0$.
We now consider the effect of multiplicative conjugation by $\f f$ on the asymptotic differential equations considered so far, as in \cite[Remark~14.4.10]{adamtt}.

\begin{lem}\label{tschirnmultconjade}
Consider the asymptotic differential equation
\begin{equation}\tag{$\f f^{-1}$\ref{eqn}}\label{eqnconj}
P_{\xf f}(Y)\ =\ 0, \hs Y \in \f f^{-1}\ca E
\end{equation}
over $K$.
Then $(\f f^{-1}\wh f, \f f^{-1}\wh{\ca E}')$ is an unraveller for
\begin{equation}\tag{$\f f^{-1}$\ref{wheqn}}\label{wheqnconj}
P_{\xf f}(Y)\ =\ 0, \hs Y \in \f f^{-1}\wh{\ca E}
\end{equation}
over $\wh K$, and
$\ddeg_{\prec 1}(P_{\xf f})_{+\f f^{-1}\wh f}=d=\ddeg_{\f f^{-1}\wh{\ca E}}P_{\xf f}$.
Moreover, if \eqref{eqnT} is a Tschirnhaus refinement of \eqref{wheqn}, then
\begin{equation}\tag{$\f f^{-1}$\ref{eqnT}}\label{eqnTconj}
\left(P_{\xf f}\right)_{+\f f^{-1}f}(Y)\ =\ 0, \hs Y \prec 1
\end{equation}
is a Tschirnhaus refinement of \eqref{wheqnconj}.
If \eqref{eqnTC} is a compatible refinement of \eqref{eqnT}, then
\begin{equation}\tag{$\f f^{-1}$\ref{eqnTC}}\label{eqnTCconj}
\left(P_{\xf f}\right)_{+\f f^{-1}(f+\wh g)}(Y)\ =\ 0, \hs Y \prece \f f^{-1}\f m
\end{equation}
is a compatible refinement of \eqref{eqnTconj}.
\end{lem}
\begin{proof}
The claims in the second sentence follow directly from Lemma~\ref{multconjade}.
The other claims are direct but tedious calculations; however, it is important to recall that $\Delta = ( \partial^{\bm i} )_{\xf f}$, so it depends on $\f f$, and, by \cite[Lemma~12.8.8]{adamtt},
\[
\big( ( \partial^{\bm i} )_{\xf f} P \big)_{\xf f}\ =\ \partial^{\bm i} P_{\xf f}.
\qedhere
\]
\end{proof}

\subsection{The Slowdown Lemma}\label{sec:redcomplex:slowdown}

In this subsection, we assume that \eqref{eqnT} is a Tschirnhaus refinement of \eqref{wheqn} and \eqref{eqnTC} is a compatible refinement of \eqref{eqnT}.
Set $\f g \coloneqq \f d_{\wh g}$, with $\wh g$ as in \eqref{eqnTC}.
The main result of this subsection is Lemma~\ref{slowdownlem}, called the Slowdown Lemma.
A consequence of this, Lemma~\ref{slowdownreducedeg}, gives the reduction to the special case of Proposition~\ref{maintechprop} considered in \S\ref{unravelspecial}.
We first prove the following preliminary lemma, which is based on \cite[Lemma~14.4.12]{adamtt} but has a different proof.
Two main differences between these settings play a role here: the change from newton polynomials to dominant parts, and the difference between the choices of $\Delta$.
Recall from \S\ref{prelim:arch} the coarsening $\prece_\phi$ of $\prece$ for $\phi \in K^{\x}$ with $\phi \not\asymp 1$.

\begin{lem}\label{slowdownaux}
Suppose that $\f f=1$.
Then
\[
\Delta P_{+f}(\wh g)\ \asymp_{\f g}\ \f g \, \Delta P_{+f}.
\]
\end{lem}
\begin{proof}
Let $f_0 \in \wh K$ be a solution of \eqref{eqnDeltaf} that best approximates $\wh f$ and satisfies $f-\wh f \sim f_0-\wh f$; in particular, $f_0 \sim f \sim \wh f \asymp 1$.
For this proof, set $Q \coloneqq \Delta P$.

Since \eqref{eqnTC} is compatible with \eqref{eqnT}, $\wh g$ is not an approximate solution of \eqref{eqnDeltaT}, and thus, with $u \coloneqq \wh g/\f g$,
\[
D_{Q_{+f, \xf g}}(\bar u)\ \neq\ 0.
\]
This yields
\[
Q_{+f}(\wh g)\ =\ Q_{+f, \xf g}(u)\ \asymp\ Q_{+f, \xf g}.
\]
Now, since $f-f_0 \prec \f g$, Lemma~\ref{adh4.5.1}(i) gives
\[
%Q_{+f}(\wh g)\ \asymp\ 
Q_{+f, \xf g}\ =\ Q_{\xf g, +f/\f g}\ \sim\ Q_{\xf g, +f_0/\f g}\ =\ Q_{+f_0, \xf g}.
\]
As $f_0$ is a solution of \eqref{eqnDeltaf}, we have
\[
\mul Q_{+f_0}\ =\ \ddeg Q_{+f_0}=1
\]
by Lemma~\ref{noasm}.
Using Lemma~\ref{asympg} and Lemma~\ref{adh4.5.1}(i) again, we get
\[
Q_{+f_0, \xf g}\ \asymp_{\f g}\ \f g \, Q_{+f_0}\ \sim\ \f g \, Q_{+f}.
\]
Finally, we obtain the desired result by combining these steps:
\[
Q_{+f}(\wh g)\ \asymp_{\f g}\ \f g \, Q_{+f}.
\qedhere
\]
\end{proof}

Using this result, we now turn to the proof of the Slowdown Lemma, based on \cite[Lemma~14.4.11]{adamtt}.
In its statement and proof, the map $vg \mapsto [vg]$ replaces $vg \mapsto v(g'/g)$ as in Lemma~\ref{reducedegunravel}; this change consequently appears also in Corollary~\ref{slowdowncor}.
The idea, as Aschenbrenner, van den Dries, and van der Hoeven note, is that ``the step from \eqref{eqn} to \eqref{eqnT} is much larger than the step from \eqref{eqnT} to \eqref{eqnTC}'' \cite[p.\ 661 or \texttt{arXiv} p.\ 565]{adamtt}.
\begin{lem}[Slowdown Lemma]\label{slowdownlem}
With $\f m$ the monomial appearing in \eqref{eqnTC}, we have
\[
\left[ v \left( \frac{\f m}{\f g} \right) \right]\ <\ \left[ v \left( \frac{\f g}{\f f} \right) \right].
\]
\end{lem}
\begin{proof}
By Lemma~\ref{tschirnmultconjade}, we may assume that $\f f=1$, so $\f m \prec f-\wh f \prece \f g \prec 1$ and $\Delta=\partial^{\bm i}$.
Set $F \coloneqq P_{+f}$ and note that $\ddeg F_{+\wh g}=\ddeg F=d$ by Lemma~\ref{adh6.6.5}(i).

\begin{claim}\label{slowdownclaim1}
$\displaystyle \f g \left(F_{+\wh g}\right)_d \prece_{\f g} \left(F_{+\wh g}\right)_{d-1}.$
\end{claim}
\begin{proof}[Proof of Claim~\theclaim]
\renewcommand{\qedsymbol}{\tiny$\blacksquare$}
By Lemma \ref{slowdownaux}, we have $\f g\, \partial^{\bm i} F \asymp_{\f g} \partial^{\bm i} F(\wh g)$,
and hence it suffices to show that
$(F_{+\wh g})_d \asymp \partial^{\bm i} F$ and
$\partial^{\bm i} F(\wh g) \prece (F_{+\wh g})_{d-1}$.

By the choice of $\bm i$, we have $\partial^{\bm i} P \asymp P$, so $\partial^{\bm i} F \asymp F_{+\wh g}$ by Lemma~\ref{adh4.5.1}(i).
As $\ddeg F_{+\wh g}=d$, we have $F_{+\wh g} \asymp (F_{+\wh g})_d$, and thus $(F_{+\wh g})_d \asymp \partial^{\bm i} F$.
By Taylor expansion, $\partial^{\bm i} F(\wh g)$ is, up to a factor from $\Q^{\x}$, the coefficient of $Y^{\bm i}$ in $F_{+\wh g}$.
Since $|\bm i|=d-1$, this yields $\partial^{\bm i} F(\wh g) \prece (F_{+\wh g})_{d-1}$.
\end{proof}

\begin{claim}
$\displaystyle \f{n \prec_n g} \implies \ddeg F_{+\wh g, \xf n} \les d-1.$
\end{claim}
\begin{proof}[Proof of Claim~\theclaim]
\renewcommand{\qedsymbol}{\tiny$\blacksquare$}
Suppose that $\f{n \prec_n g}$. Then $\f n \prec 1$, so by Corollary~\ref{adh6.6.7},
\[
\ddeg F_{+\wh g, \xf n}\ \les\ \ddeg F_{+\wh g}\ =\ d,
\]
and hence it suffices to show that $(F_{+\wh g, \xf n})_d \prec_{\f n} (F_{+\wh g, \xf n})_{d-1}$.
By Lemma~\ref{multconjhomog}, for all $i$,
\[
\left(F_{+\wh g, \xf n}\right)_i\ =\ \left(\left(F_{+\wh g}\right)_i\right)_{\xf n}\ \asymp_{\f n}\ \f{n}^i \left(F_{+\wh g}\right)_i,
\]
so we show that $\f n\, (F_{+\wh g})_d \prec_{\f n} (F_{+\wh g})_{d-1}$.
First, as $(F_{+\wh g})_d \neq 0$, we have $\f n\, (F_{+\wh g})_d \prec_{\f n} \f g\, (F_{+\wh g})_d$.
Second, $\f{n \prec g} \prec 1$ implies $[v\f g]\les[v\f n]$, so the first claim and Lemma~\ref{twocoarsen} yield $\f g\, (F_{+\wh g})_d \prece_{\f n} (F_{+\wh g})_{d-1}$.
Combining these two relations, we obtain $\f n\, (F_{+\wh g})_d \prec_{\f n} (F_{+\wh g})_{d-1}$, as desired.
\end{proof}

To finish the proof of the lemma, note that $\ddeg F_{+\wh g, \xf m}=d$, because \eqref{eqnTC} is compatible.
Then the second claim gives $\f{g \prece_m m}$, and so $\f{g \prece_g m}$ by Lemma~\ref{precgprecf}.
But since $\f{m \prec g}$, we must have $\f{m \asymp_g g}$, which means $[v\f m-v\f g]<[v\f g]$, as desired.
\end{proof}

\subsection{Consequences of the Slowdown Lemma}\label{sec:redcomplex:slowdowncor}

This first consequence, corresponding to \cite[Corollary~14.4.13]{adamtt}, follows immediately from Lemmas~\ref{getcompatibleref} and \ref{slowdownlem}.
\begin{cor}\label{slowdowncor}
If \eqref{eqnT} is a Tschirnhaus refinement of \eqref{wheqn}, then
\[
\f e \prec \wh f-f\ \implies\ \left[ v \left( \frac{\f e}{\wh f-f} \right) \right] < \left[ v \left( \frac{\wh f-f}{\wh f} \right) \right].
\]
\end{cor}
%\begin{proof}
%This follows immediately from Lemmas~\ref{getcompatibleref} and \ref{slowdownlem}.
%\end{proof}

This next consequence, corresponding to \cite[Lemma~14.4.14]{adamtt}, provides the reduction from Proposition~\ref{maintechprop} to Lemma~\ref{maintechpropspecial}.
\begin{lem}\label{slowdownreducedeg}
Suppose that \eqref{eqnT} is a Tschirnhaus refinement of \eqref{wheqn} and $\f e \prec \wh f-f$.
Let $F \coloneqq P_{+f}$, $\wh g \coloneqq \wh f-f,$ and $\f g \coloneqq \f d_{\wh g}$.
Then the asymptotic differential equation
\begin{equation}\tag{\ref{wheqn}${}_{\f g, \les d}$}\label{wheqnglesd}
F_{\les d}(Y)\ =\ 0, \hs Y \prece \f g
\end{equation}
has dominant degree $d$.
Moreover, with $\wh{\ca E}'_{\f g} \coloneqq \{y \in \wh{\ca E'} : y \prec \f g\}$, $(\wh g, \wh{\ca E}'_{\f g})$ is an unraveller for \eqref{wheqnglesd} and $\f e$ is the largest algebraic starting monomial for the unravelled asymptotic differential equation
\begin{equation}\tag{\ref{wheqn}${}'_{\f g, \les d}$}\label{wheqn'glesd}
\left(F_{\les d}\right)_{+\wh g}(Y)\ =\ 0, \hs Y \in \wh{\ca E}'_{\f g}
\end{equation}
over $\wh K$.
\end{lem}
\begin{proof}
This follows from Corollary~\ref{slowdowncor} by applying Lemma~\ref{reducedegunravel} with $\wh K$, $\wh f$, $f$, $\wh g$, $\wh{\ca E}$, $\wh{\ca E}'$, and $\wh{\ca E}'_{\f g}$ in the roles of $K$, $f$, $f-g$, $g$, $\ca E$, $\ca E'$, and $\ca E'_g$, respectively.
\end{proof}

\subsection{Proposition~\ref{maintechprop} and its consequence}
Finally, we return to the proof of the main proposition of this section.
Recall the statement:
\begin{maintechprop}
There exists $f \in \wh K$ such that one of the following holds:
\begin{enumerate}
\item\label{maintechprop-1} $\wh f-f \prece \f e$ and $A(f)=0$ for some $A \in K\{Y\}$ with $c(A)<c(P)$ and $\deg A=1$;
\item\label{maintechprop-2} $\wh f \sim f$, $\wh f-a \prece f-a$ for all $a \in K$, and $A(f)=0$ for some $A \in K\{Y\}$ with $c(A)<c(P)$ and $\ddeg A_{\x f}=1$.
\end{enumerate}
\end{maintechprop}

\begin{proof}
As noted already, if $\f{e \succe f}$, then case~\eqref{maintechprop-1} holds with $f \coloneqq 0$ and $A \coloneqq Y$, so suppose that $\f{e \prec f}$.
By Lemma~\ref{tschirnapproxsoln}, $\wh f$ is an approximate solution of \eqref{eqnDeltaf}, so by Lemmas~\ref{approxtorealade} and \ref{bestapprox}, we have a solution $f_0 \sim \wh f$ in $\wh K$ of \eqref{eqnDeltaf} that best approximates $\wh f$.
If $\wh f-a \prece f_0-a$ for all $a \in K$, then case~\eqref{maintechprop-2} holds with $f \coloneqq f_0$ and $A \coloneqq \Delta P$.
Now suppose to the contrary that we have $f \in K$ with $\wh f-f \succ f_0-f$.
That is, $f_0-\wh f \sim f-\wh f$, so in view of $f_0 \sim \wh f$, we have $f \sim \wh f$.
Hence \eqref{eqnT} is a Tschirnhaus refinement of \eqref{wheqn}.
We are going to show that then case~\eqref{maintechprop-1} holds.

If $\wh f-f \prece \f e$, then case~\eqref{maintechprop-1} holds with $A \coloneqq Y-f$, so for the rest of the proof, assume that $\f e \prec \wh f-f$, and set $F \coloneqq P_{+f}$, $\wh g \coloneqq \wh f-f$ and $\f g \coloneqq \f d_{\wh g}$.
This puts us in the situation of the previous lemma, so \eqref{wheqnglesd} has dominant degree $d$ and $(\wh g, \wh{\ca E}'_{\f g})$ is an unraveller for \eqref{wheqnglesd}.
In particular, $\ddeg_{\prec \wh g} (F_{\les d})_{+\wh g} = d$, since 
\[
\ddeg_{\prec \wh g} \left(F_{\les d}\right)_{+\wh g}\ \ges\ \ddeg_{\wh{\ca E}'_{\f g}} \left(F_{\les d}\right)_{+\wh g}\ =\ d.
\]
Also, $\f e$ is the largest algebraic starting monomial for \eqref{wheqn'glesd}.
Now since $f \in K$, we can view \eqref{wheqnglesd} as an asymptotic differential equation over $K$.
We also have $\deg F_{\les d} = d$ and $\mul (F_{\les d})_{+\wh g} < d$, since otherwise $(F_{\les d})_{+\wh g}$ would be homogeneous and so not have any algebraic starting monomials.
Thus with \eqref{wheqnglesd} in place of \eqref{eqn} and $(\wh g, \wh{\ca E}'_{\f g})$ in place of $(\wh f, \wh{\ca E}')$, Lemma~\ref{maintechpropspecial} applies.
Hence we have $g \in \wh K$ and $B \in K\{Y\}$ such that $\wh g-g \prece \f e$, $B(g)=0$, $c(B)<c(F_{\les d})$, and $\deg B=1$.
Finally, case~\eqref{maintechprop-1} holds with $f+g$ in place of $f$ and with $A \coloneqq B_{-f}$, completing the proof.
\end{proof}

In fact, if $K$ is $r$-$\d$-henselian, then the $f \in \wh{K}$ in Proposition~\ref{maintechprop} actually lies in $K$.
This follows easily from \cite[Proposition~7.5.6]{adamtt}, just as \cite[Corollary~14.4.16]{adamtt} follows from \cite[Lemma~14.1.8]{adamtt}.
We do not use Proposition~\ref{maintechprop} directly in the proof of Proposition~\ref{mainlemmadiv}, but rather this corollary concerning pc-sequences, corresponding to \cite[Corollary~14.4.15]{adamtt}.

\begin{cor}\label{maintechcor}
Suppose that $(a_\rho)$ is a divergent pc-sequence in $K$ with pseudolimit $\wh f \in \wh K$ and minimal $\d$-polynomial $P$ over $K$.
Then there exist $f \in \wh K$ and $A \in K\{Y\}$ such that $\wh f-f \prece \f e$, $A(f)=0$, $c(A)<c(P)$, and $\deg A=1$.
\end{cor}
\begin{proof}
Suppose towards a contradiction that there are no such $f$ and $A$.
Then Proposition~\ref{maintechprop} gives instead $f \in \wh K$ and $A \in K\{Y\}^{\neq}$ such that $\wh f-a \prece f-a$ for all $a \in K$, $A(f)=0$, and $c(A)<c(P)$.
Since $(a_\rho)$ has no pseudolimit in $K$, $\wh f \notin K$, and so $f \notin K$.
Hence we may take a divergent pc-sequence $(b_\sigma)$ in $K$ such that $b_\sigma \pconv f$.
Since $\wh f-b_\sigma \prece f-b_\sigma$ for all $\sigma$, we have $b_\sigma \pconv \wh f$.
The pc-sequences $(a_\rho)$ and $(b_\sigma)$ have no pseudolimit in $K$ but the common pseudolimit $\wh f \in \wh K$, and hence they are equivalent by \cite[Corollary~2.2.20]{adamtt}.
Thus $a_\rho \pconv f$, so applying Lemma~\ref{adh6.8.1} to $A$ and $f$ contradicts the minimality of $P$.
\end{proof}

\section{Proof of Proposition~\ref{mainlemmadiv}}\label{sec:mainlemmadiv}
In this section, we prove the main proposition, derived from the work of the previous sections, thus completing the proof of the main results.
Its proof is based on that of \cite[Proposition~14.5.1]{adamtt}.

\begin{mainlemmadiv}
Suppose that $K$ is asymptotic, $\Gamma$ is divisible, and $\bm k$ is $r$-linearly surjective.
Let $(a_\rho)$ be a pc-sequence in $K$ with minimal $\d$-polynomial $G$ over $K$ of order at most $r$.
Then $\ddeg_{\bm a} G=1$.
\end{mainlemmadiv}
\begin{proof}
Let $d \coloneqq \ddeg_{\bm a} G$.
We may assume that $(a_\rho)$ has no pseudolimit in $K$, as otherwise, up to scaling, $G$ is of the form $Y-a$ for some pseudolimit $a$ of $(a_\rho)$, and hence $d=1$.
We may also assume that $r \ges 1$, since the case $r=0$ is handled by the analogous fact for valued fields of equicharacteristic $0$ (see \cite[Proposition~3.3.19]{adamtt}).
By Zorn's Lemma, we may take a $\d$-algebraically maximal immediate extension $\wh K$ of $K$.
By the proof of \cite[Theorem~7.0.1]{adamtt}, $\wh K$ is $r$-$\d$-henselian.
Note that as an immediate extension of $K$, $\wh K$ is also asymptotic by \cite[Lemmas~9.4.2 and 9.4.5]{adamtt}.

Now, take $\ell \in \wh K$ such that $a_\rho \pconv \ell$, so $G$ is an element of minimal complexity of $Z(K, \ell)$ by Corollary~\ref{mindiffpolyfo}.
Lemma~\ref{ddegcutfo} gives $d \ges 1$, as well as $a \in K$ and $\f v \in K^{\x}$ such that $a-\ell \prec \f v$ and $\ddeg_{\prec \f v} G_{+a}=d$.
Towards a contradiction, suppose that $d \ges 2$.
Lemma~\ref{plimunravellersexist} then yields an unraveller $(\wh f, \wh{\ca E})$ for the asymptotic differential equation
\begin{equation}\label{eqnGav}
G_{+a}(Y)\ =\ 0, \hs Y \prec \f v
\end{equation}
over $\wh K$ such that:
\begin{enumerate}
	\item $\wh f\neq 0$;
	\item $\ddeg_{\prec \wh f} G_{+a+\wh f}=d$;\label{ddegd}
	\item $a_\rho \pconv a+\wh f+g$ for all $g \in \wh{\ca E} \cup \{0\}$;\label{pconv}
	\item $\mul G_{+a+\wh f}<d$,\label{mullesd}
\end{enumerate}
where \eqref{mullesd} follows from \eqref{pconv} by Lemma~\ref{mul}.

Suppose first that $K$ has a monomial group.
Consider the pc-sequence $(a_\rho-a)$ with minimal $\d$-polynomial $P \coloneqq G_{+a}$ over $K$.
Since $(\wh{f}, \wh{\ca E})$ is an unraveller for \eqref{eqnGav}, $\ddeg_{\wh{\ca E}} P_{+\wh f} = d > \mul P_{+\wh f}$ by \eqref{mullesd}, so let $\f e$ be the largest algebraic starting monomial for the asymptotic differential equation
\begin{equation}\label{eqnPfhat}
P_{+\wh{f}}(Y)=0, \hs Y \in \wh{\ca E}
\end{equation}
over $\wh K$ by Proposition~\ref{ndiag}.
Hence all the assumptions of the previous section are satisfied (with \eqref{eqnGav} and \eqref{eqnPfhat} in the roles of \eqref{wheqn} and \eqref{wheqn'}, respectively), so applying Corollary~\ref{maintechcor} to $(a_\rho-a)$ and $P$ yields $f \in \wh K$ and $A \in K\{Y\}^{\neq}$ such that $\wh f - f \prece \f e$, $A(f)=0$, and $c(A)<c(P)$.
Since $\f e$ is an algebraic starting monomial for \eqref{eqnPfhat}, we have $\f e \in \wh{\ca E}$, and so $f - \wh f \in \wh{\ca E}\cup\{0\}$.
But then $a_\rho-a \pconv f$ by \eqref{pconv}, so applying Lemma~\ref{adh6.8.1} to $A$ and $f$ contradicts the minimality of $P$.

Finally, we reduce to the case that $K$ has a monomial group.
Consider $\wh K$ as a valued differential field with a predicate for $K$ and pass to an $\aleph_1$-saturated elementary extension of this structure.
In particular, the new $K$ has a monomial group \cite[Lemma~3.3.39]{adamtt}.
In doing this, we preserve all the relevant first order properties: small derivation, $r$-linearly surjective differential residue field, divisible value group, asymptoticity, $r$-$\d$-henselianity of $\wh K$, and that $G \in Z(K, \ell)$ but $H \notin Z(K, \ell)$ for all $H \in K\{Y\}$ with $c(H)<c(G)$.

However, it is possible that $\wh K$ is no longer $\d$-algebraically maximal, in which case we pass to a $\d$-algebraically maximal immediate extension of $\wh K$ (and hence of $K$).
It is also possible that $(a_\rho)$ is no longer divergent in $K$, in which case we replace $(a_\rho)$ with a divergent pc-sequence $(b_\sigma)$ in $K$ with $b_\sigma \pconv \ell$.
By Corollary~\ref{mindiffpolyfo}, $G$ is a minimal $\d$-polynomial of $(b_\sigma)$ over $K$, and by Lemma~\ref{ddegcutfo}, $\ddeg_{\bm b} G = d$, where $\bm b \coloneqq c_K(b_\sigma)$.
Then our assumption $d \ges 2$ leads to a contradiction as before.
\end{proof}

\section*{Acknowledgements}
This research was supported by an NSERC Postgraduate Scholarship.
Thanks are due to Lou van den Dries for many helpful discussions and for comments on a draft of this paper, as well as to the reviewer for their careful reading of the manuscript and numerous comments and suggestions that have improved the clarity and readability of the paper.

\end{document}